\batchmode
\makeatletter
\def\input@path{{/Users/andrewpowell/Downloads/}}
\makeatother
\documentclass[oneside,english]{amsart}
\usepackage[T1]{fontenc}
\usepackage[latin9]{inputenc}
\usepackage{calc}
\usepackage{slashed}
\usepackage{amsthm}
\usepackage{amssymb}
\usepackage{stmaryrd}
\usepackage{graphicx}
\usepackage{wasysym}

\makeatletter

\providecommand{\tabularnewline}{\\}

\numberwithin{equation}{section}
\numberwithin{figure}{section}
\theoremstyle{plain}
\newtheorem{thm}{\protect\theoremname}
\theoremstyle{remark}
\newtheorem{rem}[thm]{\protect\remarkname}

\makeatother

\usepackage{babel}
\providecommand{\remarkname}{Remark}
\providecommand{\theoremname}{Theorem}

\begin{document}
\title{Topology, Cardinality, Metric Spaces and GCH}
\keywords{Baire Category, General Topology, Generalized Continuum Hypothesis,
Metric Space, Set Theory}
\subjclass[2000]{03E17 Cardinal characteristics of the continuum}
\address{Dr. Andrew Powell, Honorary Senior Research Fellow, Institute for
Security Science and Technology, Level 2 Admin Office Central Library,
Imperial College London, South Kensington Campus, London SW7 2AZ,
United Kingdom }
\email{andrew.powell@imperial.ac.uk}
\author{Andrew Powell}
\begin{abstract}
This is a paper that aims to interpret the cardinality of a set in
terms of Baire Category, \emph{i.e.} how many closed nowhere dense
sets  can be deleted from a set before the set itself becomes negligible.
To do this natural tree-theoretic structures such as the Baire topology
are introduced, and the Baire Category Theorem is extended to a statement
that a $\aleph$-sequentially complete binary tree representation
of a Hausdorff topological space that has a clopen base of cardinality
$\aleph$ and no isolated or discrete points is not the union of $<\aleph+1$-many
nowhere dense subsets for cardinal $\aleph\ge\aleph_{0}$, where a
$\aleph$-sequentially complete topological space is a space where
every function $f:\aleph\rightarrow\{0.1\}$ is such that $(\forall x)(x\in f\rightarrow x\in\in X)\rightarrow(f\in X)$.
It is shown that if $\aleph<\left|X\right|\le2^{\aleph}$ for $\left|X\right|$
the cardinality of a set $X$, then it is possible to force $\left|X\right|-\aleph\times\left|X\right|\ne\emptyset$
by deleting a dense sequence of $\aleph$ specially selected clopen
sets, while if any dense sequence of $\aleph+1$ clopen sets are deleted
then $\left|X\right|-(\aleph+1)\times\left|X\right|=\emptyset$. This
gives rise to an alternative definition of cardinality as the number
of basic clopen sets (intervals in fact) needed to be deleted from
a set to force an empty remainder. This alternative definition of
cardinality is consistent with and follows from the Generalized Continuum
Hypothesis (GCH), which is shown by exhibiting two models of set theory,
one an outer (modal) model, the other an inner, generalized metric
model with an information minimization principle.
\end{abstract}

\maketitle

\section{Introduction}

This paper is experimental in the sense that there is very little
recent relevant literature in the subject of this paper, and the paper
has not been peer reviewed. For these reasons, please email the author
if you find any errors or any arguments lack clarity.\\
\\
In this paper a natural topology is outlined on the natural numbers,
real numbers and sets higher up in the von Neumann cumulative hierarchy
of pure sets\footnote{\cite{key-6} is used as the standard reference for motivating the
axioms of set theory (Zermelo-Fraenkel set theory with the Axiom of
Choice) and \cite{key-2-0}is the standard reference for developments
in Zermelo-Fraenkel set theory.}, which leads to a change in the definition of cardinality of set
in order to support the view that cardinality measures how many topologically
negligible sets can be deleted from a set before the set itself becomes
negligible. It is then shown that there are models of set theory in
which the change of definition of cardinality can be performed (which
is exactly when the Generalized Continuum Hypothesis, GCH, holds).
\\
\\
Before we begin with the development of a natural topology, it is
worth noting some assumptions about the universe of sets. Firstly,
we identify the set of subsets of a set $X$ of cardinality $\aleph$
with the set of binary sequences of length $\aleph$, called \emph{binary
$\aleph$-sequences}, which are functions $f:\aleph\rightarrow\{0.1\}$,
and members of the functions $\langle\alpha,b\rangle$ are called
\emph{nodes}. This is possible by fixing an enumeration of $X$,  $\langle x_{\alpha}:\alpha<\aleph\rangle$
(by the Axiom of Choice), and for any subset $Y\subseteq X$ forming
the binary \emph{$\aleph$-}sequence $\langle b_{\alpha}:(x_{\alpha}\in Y\rightarrow b_{\alpha}=1)\vee(x_{\alpha}\notin Y\rightarrow b_{\alpha}=0)\rangle$.
Thus a subset of $X$ can be identified with a binary $\aleph$-sequence,
and a set of subsets of $X$ can be identified with a set of binary
$\aleph$-sequences. It is natural to think about any set as a \emph{tree}
of binary $\aleph$-sequences for some cardinal $\aleph$, where subtrees
may \emph{split} from a given $\aleph$-sequence at a given node,
if we allow a tree to include the degenerate case where all members
of the set are subsets of a single branch of the tree (\emph{i.e}.
the tree is a line). It is an obvious but important fact that a tree
formed by binary $\aleph$-sequencs is a \emph{binary tree}, \emph{i.e}.
a tree in which every node has at most two successor nodes. \\
\\
Representation by binary trees also suggests a property of sets that
will appear throughout this paper, namely the property of a set $X$
corresponding to every binary $\aleph$-sequence through the tree
representing a member of $X$. This property is a kind of completeness,
but is in general weaker than compactness (unless $\aleph\le\aleph_{0}$).
It is called \emph{$\aleph$-sequential completeness}. Logically \emph{$\aleph$}-sequential
completeness has the form\linebreak{}
 $(\forall f:\aleph\rightarrow\{0,1\})((\forall x)(x\in f\rightarrow x\in\in X)\rightarrow(f\in X))$,
where $x\in\in y$ is defined as $(\exists z)(x\in z\wedge z\in y)$.
Like completeness in a metric space, $\aleph$-sequential completeness
does correspond to a generalized metric condition. \emph{$\aleph$-}sequential
completeness is also a closure condition, but it is stronger than
closure because closure depends on which sets are defined to be open.
It is in fact a form of absolute closure\footnote{\cite{key-4-0} is a reference for a notion of \emph{absolute closure},
defined as 'A Hausdorff space $X$ is called absolutely closed if
$X$ is closed in every Hausdorff space in which it is imbedded''
(see \cite{key-4-0} Definition 1.1).}, because the closure does not depend on the embedding space.\footnote{\emph{$\aleph$-sequential completeness} is not the same as \emph{compactness}
(because for $\aleph>\aleph_{0}$ there are infinite covers without
finite subcovers) or \emph{sequential compactness} (because $\aleph$-sequences
are longer than $\aleph_{0}$-sequences for $\aleph>\aleph_{0}$).
In particular $\aleph$-sequential completeness is not the same as
the Stone-\v{C}ech compactification, because $\aleph$-sequentially
complete Baire spaces, unlike Stone spaces over an infinite power
set, are not compact for $\aleph>\aleph_{0}$, and Stone spaces over
an infinite power set are larger ($=2^{2^{\aleph}}$) and richer than
the power set with an $\aleph$-sequentially complete Baire topology
(cardinality $=2^{\aleph}$), see \cite{key-1-0} Theorem 4.3 p. 143
\& p. 146. \cite{key-8} and the online $\pi$-base resource give
an excellent view of the landscape of the topological properties of
topological spaces.}\\
\\
We can also note that by the same argument as above any set can be
considered as a (possibly infinitely long) binary sequence. An ordinal
number, $\alpha$, can be coded (non-uniquely) as a constant sequence
of 1s of length $\alpha$, but in order to associate $\alpha<\aleph$
with a unique binary $\aleph$-sequence, $\alpha$ is represented
as an initial sequence of $\alpha$ 1s followed by a terminal sequence
of 0s. $x\subseteq y$ if $x_{\alpha}\le y_{\alpha}$ for all $\alpha<\aleph$
where $x_{\alpha}$ and $y_{\alpha}$ are binary representations at
position $\alpha$ in a $\aleph$-sequence of sets $x$ and $y$.\\
\\
 It should be apparent that the universe of sets can be regarded as
a binary sequence representation of the von Neumann hierarchy of pure
sets, $V_{0}=\emptyset$, $V_{\alpha+1}=\{x:x\subseteq V_{\alpha}\}$
and $V_{\lambda}=\bigcup_{\alpha<\lambda}V_{\alpha}$ for $\lambda$
a limit ordinal. Binary\emph{ }$\aleph$-sequences first appear in
$V_{\alpha}$ for some ordinal $\alpha$, and if $\aleph$ is an infinite
cardinal then $\alpha>\omega$, where $\omega$ is the least ordinal
of cardinality $\aleph_{0}$.

\section{A natural topology of the natural numbers}

Consider a topology on the natural numbers with closed sets of the
form $u_{n}=\{m\in N:m>n\}$ as well as \emph{$\slashed{O}$ }and
\emph{N}, where $N$ is the set of natural numbers. These sets are
closed sets because $\bigcap_{m<i<n}u_{i}=u_{n}$, $\bigcap_{m<i<\omega}u_{i}=\slashed{O}$,
$u_{n}\cap N=u_{n}$, $u_{n}\cap\slashed{O}=\slashed{O}$ and $N\cap\slashed{O}=\slashed{O}$
for natural numbers $m,\:n$ and $n>m+1$ where both appear in the
same formula. No new closed sets are introduced by taking finite unions
of closed sets, \emph{i.e}. $\bigcup_{i\in\{m_{0},\dots,m_{n}\}}u_{i}=u_{m_{0}}$
if $m_{0}\le\ldots\le m_{n}$ for $m_{i}\in N$. Rephrasing these
statements, it is easy to see that $u_{j}\subset u_{i}$ if $j>i$
and for any natural numbers $m,\:n>m+1$, $\bigcap_{m<i<n}u_{i}\neq\slashed{O}$
and $\bigcap_{m<i<\omega}u_{i}=\slashed{O}$. Define open sets to
be $d_{n}=N-u_{n}$ and $\slashed{O}$ and $N$. Then we see $N-\bigcup_{m<i<n}d_{i}\ne\slashed{O}$
and $N-\bigcup_{m<i<\omega}d_{i}=\slashed{O}$ for any natural numbers
$m,\:n>m+1$. If we note $\mid d_{n+1}\mid-\mid d_{n}\mid=1$, then
we have $\left|N\right|\ne\sum_{i=m}^{n}$1 and $\left|N\right|=\sum_{i=m}^{\omega}1$,
or in cardinality terms $\aleph_{0}\ne n-m$ and $\aleph_{0}=\mid\omega\mid\times1$.
These results are not surprising, but it is worth rephrasing: that
if we remove any finite set of open sets (not including $N$) from
$N$ we have a non-empty remainder, and if we remove any infinite
set of open sets from $N$ we have an empty remainder. This shows
in this topology that you cannot force\footnote{``Force'' is used in its everyday sense,\emph{ i.e.} it is possible
with some effort to do something. The powerful mathematical notion
of forcing, which amounts to giving a set a property by adding a consistent
set of finitely specified conditions, is related to developments later
in this paper. (See \cite{key-3} for a clear introduction to set-theoretic
forcing). } $\aleph_{0}$ to be finite and you cannot force $\aleph_{0}\ne\mid\omega\mid$.
\\
\\
We can state this as:
\begin{thm}
In a natural topology on the set of natural numbers, $N$, with closed
sets of the form $u_{n}=\{m\in N:m>n\}$ and \emph{$\slashed{O}$
}and \emph{$N$, }you cannot force $\aleph_{0}$ to be finite and
you cannot force $\aleph_{0}\ne\mid\omega\mid$.
\end{thm}

It is possible to reverse the roles of the open and closed sets, but
essentially the same topology arises. If, however, (other than closed
and open sets $\emptyset$ and $N$) closed sets have the form $u_{n}=\{m\in N:m\le n\}$
and open sets have the form $d_{n}=\{m\in N:m<n\}$, then all open
sets are also closed, and all closed sets also open because $m\le n$
if any only if $m<n+1$. But then each set of the form $\{n\}$, where
$n\in N$, and $\slashed{O}$ and $N$ are clopen (open and closed).
It follows that $X$ has the discrete topology.

\section{A natural topology of the real numbers}

A natural generalization of taking terminal segments of the natural
numbers as closed sets is to take closed sets of the real numbers
to be the set of real numbers in a set $X$ (expressed as a set of
binary sequences) that agree with some $x\in X$ on a particular initial
segment of $x$, say $\langle x_{m}:m<n\rangle$ where $x_{m}\in\{0,1\}$,
but which do not include $x_{n}$ and therefore $x$. To be precise,
a closed set is a set of the form $u_{n}(x)=\{y\in X:$$y_{n}\ne x_{n}\wedge(\forall m<n)(x_{m}=y_{m})\}$
for natural number $n>0$. This is a variation of the Baire topology,
where open sets extend a finite sequence $\langle x_{m}:m<n\rangle$.\footnote{See \cite{key-4} s. 2.1, p. 222, for example where the finite (binary
for definiteness) sequence $s$ is the initial segment of $x$. The
case of binary sequences gives rise to the Cantor topology.} We can see that other closed sets need be added so that closed sets
are closed under intersection. We need to add as closed sets\emph{
$\slashed{O}$ }and \emph{X} and sets of the form $\{x\}$ for $x\in X$
because $\bigcap_{1<i<\omega}u_{n_{i}}(x_{i})=\{x\}$ is possible
for some sequence of closed sets $\langle u_{n_{i}}(x_{i}):1<i<\omega\rangle$.\footnote{An easy way to see that $\bigcap_{1<i<\omega}u_{n_{i}}(x_{i})=\{x\}$
is possible is to use the tree approach in the main text: at the $i$-th
split in the binary tree that represents $X$ choose a subtree which
contains a member of $u_{n_{i}}(x_{i})$, choosing one subtree (using
the Axiom of Choice, or choosing the one that starts with 0 if you
want to avoid the Axiom of Choice) if there is a choice. Then the
sequences of choices defines a path $x$ which may be in $\bigcap_{1<i<\omega}u_{n_{i}}(x_{i})$,
and will be if $x\in X$ because $x\in u_{n_{i}}(x_{i})$ for each
$1<i<\omega$. Conversely, for any $x\in X$ it is always possible
to construct a sequence of $u_{n_{i}}(x_{i})$ such that $\bigcap_{1<i<\omega}u_{n_{i}}(x_{i})=\{x\}$,
by choosing $u_{n_{i}}(x_{i})$ such that $x\in u_{n_{i}}(x_{i})$
at each split. } The topology can be thought of in terms of trees: if each member
of $X$ is a binary sequence $\omega\rightarrow\{0,1\}$, \emph{i.e.}
a binary \emph{$\omega$-sequence}, then every $x\in X$ is a \emph{branch}
of the tree and $u_{n}(x)$ is a subtree that splits from $x$ at
a particular node of the binary sequence $x$, \emph{i.e.} at $\langle n,b\rangle$
for natural number $n$ and $b\in\{0,1\}$. It is possible for a point
to be represented by two branches in the case where a binary sequence
is eventually constant. For example, $0111\ldots$ might represented
by $1000\ldots$ as well. But since there are countably many such
double representations, the use of a tree representation is appropriate
for studying uncountable sets of real numbers. In the following treatment
all \emph{isolated} members of $X$, $x\in X$ which have a highest
node $nd$ such that all other $y\in X$ split from $x$ at or below
$nd$, are deleted to simplify the exposition. As there are at most
countably many isolated members of $X$ because there are countably
nodes of type $nd$, isolated points are well-understood from a cardinality
perspective. Moreover, since it is possible for a point to become
isolated if other isolated points are removed, by transfinite induction
up to a countably infinite ordinal (deleting any $x\in X$ that is
covered only by isolated points of order $<\alpha$ at limit ordinal
$\alpha$), we can delete all isolated points and leave either the
empty set or a dense-in-itself kernel\footnote{A set is \emph{dense-in-itself} if it contains no isolated points.
The word ``kernel'' indicates that the isolated points have been
removed and that a non-empty set remains. The construction is from
\cite{key-1} p. 198.} of the set.\\
\\
It is also possible to remove closed sets of the form $u_{n}(x)$
in a way which is a generalized version of the construction of a Cantor
ternary set. Fix an enumeration of countably infinitely many $x\in X$,
say $\langle x_{\alpha<\omega}\rangle$, treated as branches of binary
sequences of length $\omega$, which are dense in $X$, \emph{i.e.}
every $x\in X$ is covered by an $\omega$-sequence of $x_{\alpha}$'s\footnote{This definition is equivalent to the standard definition of every
open neighbourhood of $x$ having non-empty intersection with the
dense set.}. Then for any finite ordinal $\alpha$ there is a highest node of
height $n(\alpha)$ at which $x_{\beta<\alpha}$ split from $x_{\alpha}$
(where $n(0)$ is the lowest node from which some branch splits from
$x_{0}$); then proceed along the branch $x_{\alpha}$ $r(\alpha)>0$
nodes from which some $u_{n(\alpha)+r(\alpha)}(x_{\alpha})$ splits,
and delete any branches that coincide with the terminal segment of
$x_{\alpha}$ from nodes of height $r(\alpha)+1$ onwards. If $x_{\alpha}$
has already been deleted then do nothing. For reference we will call
this branch deletion construction from set $Y\subseteq X$ $cntr(Y;x_{\alpha})$.
Finally, at ordinal $\omega$ take the intersection of all stages
of the construction $\alpha<\omega$. We can write the construction
$X_{0}=X$, $X_{\alpha+1}=cntr(X_{\alpha}x_{\alpha})$ for $\alpha<\omega$
and $X_{\omega}=\bigcap_{\alpha<\omega}X_{\alpha}$. \\
\\
\includegraphics[scale=0.5]{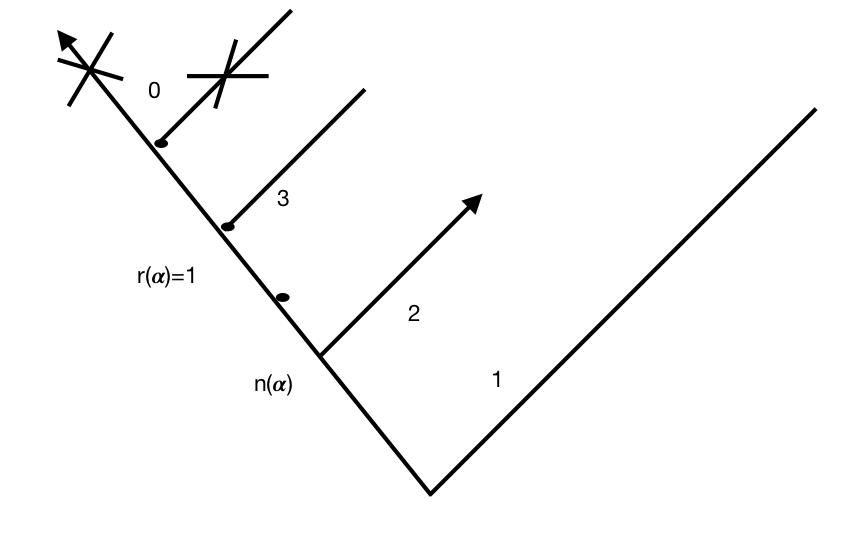}\\
\\
\emph{Figure 1:} \emph{Subtrees (clopen intervals) deleted from a
binary tree representation of a set. The construction proceeds $r(\alpha)$
nodes that have branches splitting from them (one empty node is skipped
in the diagram) them where $x_{\alpha-1}$ splits from $x_{\alpha}$
and deletes nodes $r(\alpha)+1$ and higher of $x_{\alpha}$.} \\
\\
The density of the sequence $\langle x_{\alpha<\omega}\rangle$ in
$X$ ensures that $X_{\alpha<\omega}\ne\emptyset$ because each non-empty
closed set $u_{n(\alpha)+m}(x_{\alpha})$ for $1\le m\le r(\alpha)$
will contain some $x_{\beta>\alpha}.$ The resulting Cantor sets,
$X_{\omega}(\langle x_{i<\omega}\rangle)$, are closed and nowhere
dense.\footnote{If you proceed along the branch $x_{\alpha}$ by exactly $1$ node,
then the intersection will form a single branch, \emph{i.e.} contain
exactly one point.} The reason why the Cantor sets are closed is that $u_{n}(x_{\alpha})$
are closed and open (clopen) since $u_{n}(x_{\alpha})$ contains all
of its limits points in $X$, and $X-u_{n}(x_{\alpha})$ contains
all of its limit points (so both are clopen), and the Cantor sets
constructed at ordinal $\omega$ have the form $X-\bigcup_{\alpha<\omega}u_{n(\alpha)+r(\alpha)}(x_{\alpha})$,
$i.e.$ the complement of an open set.\footnote{It is worth noting that $\bigcup u_{n}(x_{\alpha})$ is not in general
closed, because all $u_{n}(x_{\alpha})$ could split from a common
sequence that is $\notin X$. } In this paper $u_{n}(x)$ are known as \emph{clopen intervals}. To
see that a resulting Cantor set is nowhere dense, note that each clopen
interval is a maximally dense subset of itself, and since each clopen
interval in the tree has a clopen interval deleted from it because
the sequence $\langle x_{i<\omega}:x_{i}\in X_{\alpha}\rangle$ is
dense in $X_{\alpha}$, no subset of the Cantor set is dense in the
tree. \\
\\
While the tree model of a set of real numbers, $X$, is a strong visual
construction, there is a case where the model is not applicable, namely
where all $x\in X$ cover a single \emph{$\omega$}-sequence. In this
case, there are no clopen intervals splitting from the single \emph{$\omega$}-sequence.
The members of $X$ will either have a finite length, written $x_{n}$
for $n\in\omega$, or be an $\omega$-sequence, $x_{\omega}$. Each
$\{x_{n}\}$ is a closed set because $\{x_{n}\}$ contains its limit
points, since it contains only one point, and $X-\{x_{n}\}$ is closed
as its closure is $X-\{x_{n}\}$. Hence $\{x_{n}\}$ is clopen. $\{x_{\omega}\}$
is also closed because it contains its limit point, but is not open
(as $X-\{x_{\omega}\}$ has $x_{\omega}$ as a limit). But given that
all $\{x_{n}\}$ are isolated because they are discrete sets, we can
remove them, and leave the set $\{x_{\omega}\}$ or the empty set.
If $\{x_{\omega}\}$ exists, then it too is isolated because it is
now a clopen and therefore is a discrete set (as its complement is
the empty set). As $\left|X\right|\le\aleph_{0}$ and $X$ comprises
isolated points, this case has been sufficiently characterized from
a cardinality perspective.\\
\\
We can state this as:
\begin{thm}
In the Baire topology on a subset of the set of real numbers, X, that
comprises a set of binary sequences with a countable basis of clopen
intervals and no discrete or isolated points, the Cantor sets $X_{\omega}$
constructed from $X$ and a dense $\omega$-sequence $\langle x_{\alpha<\omega}\in X\rangle$
by $X_{0}=X$, $X_{\alpha+1}=cntr(X_{\alpha};x_{\alpha})$ for $\alpha<\omega$
and $X_{\omega}=\bigcap_{\alpha<\omega}X_{\alpha}$ are closed and
nowhere dense. 
\end{thm}

In terms of cardinality, note that $X_{\omega}$ can be empty, but
if $X$ is an uncountable \emph{sequentially complete}\footnote{Sequential completeness is not a topological notion as the real numbers
with the standard open interval topology is homeomorphic to the open
interval\emph{ $(0,1)$ }with the same open interval topology, but
the real numbers is sequentially complete and\emph{ $(0,1)$ }is not
because the constant 0 and constant 1\emph{ $\omega$}-sequences are
sequences through the tree but 0 and 1 are not members of\emph{ }$(0,1)$.
Sequential completeness is a metric notion in general, but in the
case of the Baire topology it is also set-theoretic (whether any given
sequence is a member of a set) and tree-theoretic (whether a sequence
is covered by other sequences that split from it at a node implies
that the sequence is a member of the set). There is little difference
in practice because a Baire space is metrizable with, for example,
the metric $d(x,x)=0$ and $d(x,y):=2^{-n}$ where \emph{n} is the
height of the lowest node such that $(x)_{n}\ne(y)_{n}$.} set of real numbers (in the sense that every \emph{$\omega$}-sequence
through the tree is a member of the set), then $X$ has cardinality
$2^{\aleph_{0}}$. This is so because every uncountable set of binary
sequences must cover an infinite binary tree (because each node is
covered by a binary sequence). Remove all isolated binary sequences.
Then each binary sequence must split at an arbitrarily high node (\emph{i.e.}
into 0 and 1), since otherwise the binary sequence would be an isolated
point; and by sequentially completeness, the tree created is isomorphic\footnote{That is, there is a one-to-one mapping of $X$ onto $2^{\omega}$
that preserves the branch structure.} to the set of all binary $\omega$-sequences, $2^{\omega}$, which
has cardinality $2^{\aleph_{0}}$. The subtree generated by the closed
nowhere dense set construction, $X_{\omega}$, also has cardinality
$2^{\aleph_{0}}$, as can be seen by labelling the remaining $u_{n}(x_{\alpha})$
1,2 \emph{et seq} and the deleted subtrees 0 and noting that nested
$u_{n}(x_{\alpha})$ give rise to sequences that can be labelled using
$\omega$-sequences that do not contain 0. We can conclude by sequential
completeness that each sequence is a member of $X$. In cardinality
terms we have $2^{\aleph_{0}}-\aleph_{0}\times2^{\aleph_{0}}=2^{\aleph_{0}}\ne\slashed{O}$.
\\
\\
\includegraphics[scale=0.3]{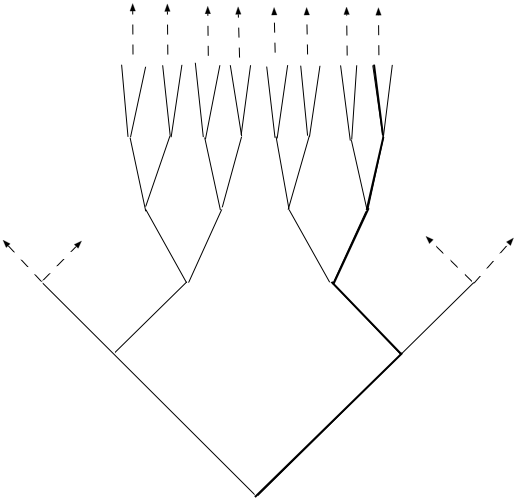}\\
\\
\emph{Figure 2: A tree representation of a set of binary sequences.
The bold line is a path through the tree. A set is sequentially complete
if every path through the tree is a member of the set.}\\
\\
We can state this as:
\begin{thm}
In a sequentially complete Baire topology on a set of real numbers,
$X$, that comprises a sequentially complete set of binary sequences
with a countable basis of clopen intervals and no discrete or isolated
points, the Cantor sets have cardinality $2^{\aleph_{0}}$ and the
process of deleting $\omega$ clopen intervals gives rise to the equation
$2^{\aleph_{0}}-\aleph_{0}\times2^{\aleph_{0}}=2^{\aleph_{0}}\ne\slashed{O}$.
\end{thm}

On the other hand, because there are only countably infinitely many
clopen intervals (since there are only countably infinitely many nodes
from which clopen intervals split from a branch), if we were to delete
$\aleph_{1}$ clopen intervals in a dense way\footnote{That is, there is no non-empty subset of $X$ that does not have a
clopen interval deleted from it. If there were some clopen interval
which were not subject to deletion, then the empty set would not result.} the empty set would result. This may seem meaningless, but we can
say that every $\omega_{1}$-sequence of Cantor sets, $C_{\alpha}$,
such that $C_{\beta}\subset C_{_{\gamma}}$ if $\beta>\gamma$, has
a terminal segment of empty sets, \emph{i.e}. $C_{\delta}=\emptyset$
for all $\delta>\beta$ for some countable ordinal $\beta$.\footnote{It is always possible to re-order any countably infinite set as a
total ordering of order type $\alpha$ for any $\alpha<\omega_{1}$,
but if there were a strictly decreasing nested sequence of Cantor
sets of length $\aleph_{1}$ then as at least one clopen interval
is deleted at each step, $X$ would have at least $\aleph_{1}$ clopen
intervals, which is false. Likewise a tree with a path of length $\omega_{1}$
with $\omega_{1}$ clopen intervals splitting from it would define
a set of clopen intervals of cardinality $\aleph_{1}$, which does
not exist; and thus no tree which has a path of length $\omega_{1}$
with $\omega_{1}$ clopen intervals splitting from it represents a
set of real numbers.} In cardinality terms we have $2^{\aleph_{0}}-\aleph_{1}\times2^{\aleph_{0}}=\slashed{O}$.
Moreover, although we can force $2^{\aleph_{0}}-\aleph_{0}\times2^{\aleph_{0}}=\slashed{O}$
(delete any $\omega$-sequence of all clopen intervals $\subset X$
from $X$), we cannot force $2^{\aleph_{0}}-\aleph_{1}\times2^{\aleph_{0}}\ne\slashed{O}$
as the deletion of any dense uncountable sequence of clopen intervals
will result in an empty remainder.\footnote{If the denseness condition were removed, it would be possible to delete
the same clopen interval $\aleph_{1}$ times, or rather delete it
once and then do nothing $\aleph_{1}$ times.} \\
\\
We can state this as:
\begin{thm}
\label{thm:In-a-sequentially}In a sequentially complete Baire topology
on a set of real numbers, X, that comprises a sequentially complete
set of binary sequences with a countable basis of clopen intervals
and no discrete or isolated points, if $\aleph_{1}$ clopen intervals
are deleted in a dense way, then the empty set results, i.e. $2^{\aleph_{0}}-\aleph_{1}\times2^{\aleph_{0}}=\slashed{O}$.
\end{thm}

There is also a connection between this topology and the Baire Category
Theorem for compact\footnote{A topological space  $X$ is \emph{compact} if for every set of subsets
$M\subseteq N$ such that $\bigcup M=X$ there is a finite set of
subsets $L\subseteq M$ such that $\bigcup L=X$. A topological space
$X$ is locally compact if every $x\in X$ has some open set $U$
and some compact set $C$ such that $x\in U\subseteq C$, } Hausdorff\footnote{A topological space  is \emph{Hausdorff} if there are disjoint neighbourhoods
around any two distinct points, \emph{i.e. }$Hausdorff(\langle x,N\rangle):=(\forall x\in X)(\forall y\in X)(\exists Y\in N)(\exists Z\in N)(x\neq y\rightarrow Y\cap Z=\slashed{O})$.} topological spaces, \emph{i.e.} that a compact Hausdorff topological
space is not the union of countably many closed nowhere dense subsets.
A topological space that comprises a sequentially complete set of
binary sequences with countably infinitely many clopen basis sets
$u_{n}(x)$\footnote{A space that has a base of countably many open sets is called \emph{second-countable}.}
and no discrete or isolated points, $2^{\omega}$ for short\footnote{A standard Baire space, $\omega^{\omega}$, is sequentially complete,
but it is not compact nor sequentially compact, in essence because
it is too wide: it has unbounded sequences of branches and an infinite
cover comprising those branches and subtrees that split from them
that does not have a finite subcover. The notation $2^{\omega}$ reflects
the fact that the topological space is actually a Cantor space, \emph{i.e.}
$[0,1]$ with the Baire topology. A Cantor space is compact (as a
product of a compact set, namely $2=\{0,1\}$. }, is compact and Hausdorff.\footnote{\label{fn:To-show-compactness}To show compactness from first principles,
proceed using a Heine-Borel construction. Assume that a topological
space that comprises a sequentially complete set of binary sequences
with countably infinitely many clopen basis sets and no discrete or
isolated points, $X$, is not compact, \emph{i.e.} there exists an
infinite cover of open sets $\{C_{i<\alpha}:\alpha\ge\omega\}$ without
a finite open subcover. Then subdivide the set underlying of $X$,
$E$ say, into two disjoint clopen intervals $E_{1}$ and $E_{2}$
(which exists since $X$ has a clopen basis and the complement of
any clopen set is clopen) and iterate the process. At least one clopen
interval in each subdivision will not be compact. Because \emph{$E$}
is a sequentially complete and has a countably infinite basis, if
a nested sequence $E_{N}$ consists of closed non-empty sets, where
\emph{N} is an $\omega$-sequence of finite binary sequences such
that if $m\in N$ and $n>m$ and $n\in N$ then \emph{m} is a subsequence
of \emph{n}, the subdivision process $\bigcap_{n\in N}E_{N}$ will
result in a non-empty set. In fact, because $N$ defines a unique
point, $\bigcap_{n\in N}E_{N}$ contains exactly one point in $E,$$L$.
Now every point in $E$ will be a member of at least one open set,
$C_{j}$, in the cover, otherwise $E$ would not be covered. But $L\in E_{n}\subseteq C_{j}$
for some $n\in N$ where $E_{n}\in E_{N}$ since an open set $C_{j}$
such that $L\in C_{j}$ will include a clopen interval $E_{n}$ such
that $L\in E_{n}$ because the space has a basis of clopen intervals.
By construction any clopen interval will split from $E$ depending
only on its first $n$ binary digits for some natural number $n$.
Thus if an $E_{N}$ were a nested sequence of clopen intervals that
are not compact, then we would have $L\in E_{m}\subset E_{n}\subseteq C_{j}$
for all $m>n$ where $E_{m}\in E_{N}$ for any $E_{m}$ whose members
agree with members of $C_{j}$ on the first $n$ binary digits, which
means that $\{C_{j}\}$ is a single (\emph{i.e.} finite) cover for
$E_{m>n}$, contradiction. To show the space is Hausdorff, note that
any two distinct branches \emph{c} and \emph{d} will split from one
another at a certain node, $nd=\langle n,b\rangle$ where $b\in\{0,1\}$:
a $u_{n}(d_{\alpha})$ that includes \emph{c} and all branches that
split from \emph{c} after node \emph{n} will have a disjoint union
with a $u_{n}(d_{\alpha})$ that includes \emph{d} and all branches
that split from \emph{d} after node \emph{n}. Hence the space is Hausdorff.} In cardinality terms the Baire Category Theorem implies that $2^{\aleph_{0}}\ne\aleph_{0}\times2^{\aleph_{0}}$
given that each closed nowhere dense set in the compact Hausdorff
topological space $2^{\omega}$ has cardinality $2^{\aleph_{0}}$.\\
\\
We can state this as:
\begin{thm}
In a Hausdorff topological space, X, that comprises a sequentially
complete set of binary sequences with a countable clopen base and
no discrete or isolated points, X is not the union of countably many
nowhere dense subsets.
\end{thm}

It is also worth noting that we do not need to start with a sequentially
complete set $X$ with the Baire topology\footnote{If $X$ is not compact or sequentially complete, the Baire Category
Theorem does not apply.}. If $X$ is not sequentially complete, contains no sequentially complete
clopen interval, has a dense-in-itself subset and has cardinality
$\aleph_{0}<c\le2^{\aleph_{0}}$ such that all clopen sets have cardinality
$c$ (removing all clopen intervals of cardinality $<c$ if necessary),
then by removing clopen intervals in a dense way following Theorem
\ref{thm:In-a-sequentially} we see that $c-c\times\aleph_{1}=\slashed{O}$.
It is in fact possible using the Cantor construction $X_{1}=X$, $X_{\alpha+1}=cntr(X_{\alpha})$
for $\alpha<\omega$ and $X_{\omega}=\bigcap_{\alpha<\omega}X_{\alpha}$
to construct an $X_{\omega}$ which contains any given $x\in X$ by
choosing the set $\langle x_{\alpha<\omega}\rangle$ of $\omega$-sequences
to be deleted such that $x_{\alpha<\omega}\in X$, $x_{\alpha}\ne x$
and $\langle x_{n<\omega}\rangle$ is dense in $X$,\footnote{There are at least $\aleph_{0}$ such $\omega$-sequences because
if there were a finite number, then some clopen interval would be
sequentially complete. } and by modifying $cntr$ to increase the value of $r(\alpha)$ so
that $x\in X_{\alpha}$ for all $\alpha<\omega$ and therefore $x\in X_{\omega}$
by definition.\\
\\
 If we consider that each clopen interval is divided into $r>1$ disjoint
clopen intervals and one clopen interval is deleted, we can write
$U_{\alpha}=\bigcup_{0\le m\le r(\alpha)}U_{\alpha,m}$ and set $U_{\alpha+1}=U_{\alpha,m}$
for any choice of $m$ such that $1\le m\le r(\alpha)$, where $U_{\alpha,m}$
are clopen intervals, $U_{\alpha,0}$ is deleted because $x_{\alpha}$
is a branch in $U_{\alpha,0}$, $U_{1}=X$ and $U_{\omega}=\bigcap_{\alpha<\omega}U_{\alpha}$,
for natural numbers $\alpha,\:m,\:r(\alpha)$. Then if $U_{\alpha}$
preserves $y_{\alpha}\in U_{\alpha}$, \emph{i.e.} $y_{\alpha}\in U_{\omega}$,
then if $y_{\alpha}\in U_{\alpha,m}$ for some $m>0$ there are $y_{\alpha,s}\in U_{\alpha,s}\ne y_{\alpha}$
for all $1\le s\le r(\alpha)$ and $s\ne m$. We require that $U_{\alpha}$
is constructed to include a clopen interval around branch $x_{\alpha}\in U_{\alpha,0}$
(which will be deleted), to preserve $\bigcup_{1\le m<\alpha}\{y_{m}\}$
and $y_{\alpha}\in U_{\alpha,m}$ for some $m>0$. This requirement
can be met by selecting $y_{1}\in U_{1}$ such that $y_{1}\ne x_{\beta}$
for any $\beta<\omega$ and constructing $U_{\alpha+1}$ and $y_{\alpha+1}$
as follows given clopen interval $U_{\alpha}$ and $y_{\alpha}\in U_{\alpha}$
which is preserved, $i.e.$ $y_{\alpha}\in U_{\omega}$:\footnote{The rate of growth of $r(n)$ depends on the height of the splitting
node of $x_{n}$ and $y_{n}$, which could be set arbitrarily high. }\\
\\
\noindent\fbox{\begin{minipage}[t]{1\columnwidth - 2\fboxsep - 2\fboxrule}%
\begin{itemize}
\item If $x_{\alpha}$ is a branch in $U_{\alpha}$: set $r(\alpha)$ to
include the node where $y_{\alpha}$ splits from $x_{\alpha}$, set
$U_{\alpha,0}\subset U_{\alpha}$ to be a clopen interval such that
$x_{\alpha}$ is a branch in $U_{\alpha,0}$ and $y_{n}\notin U_{\alpha,0}$,
and set $U_{\alpha+1}:=U_{\alpha,m}$ for any choice of $1\le m\le r(\alpha)$
where $U_{\alpha}=\bigcup_{0\le m\le r(\alpha)}U_{\alpha,m}$. Set
$y_{\alpha+1}:=y_{\alpha}$ if $y_{\alpha}\in U_{\alpha+1}$ and otherwise
choose $y_{\alpha+1}\in U_{\alpha+1}$ such that $y_{\alpha+1}\ne y_{1\le m\le\alpha}$
(which is possible since each clopen interval such as $U_{\alpha+1}$
will have uncountably infinitely many members)\footnote{The same construction works for $U_{\alpha+1}$ having countably infinitely
many members as $\alpha+1<\omega$ }.\\
\item If $x_{\alpha}$ is not a branch in $U_{\alpha}$: set $U_{\alpha+1}:=U_{\alpha}$
and set $y_{\alpha+1}:=y_{\alpha}$.
\end{itemize}
\end{minipage}}\\
\\
\\
Since each $y_{\alpha+1}$ is preserved by the same construction as
was used for $y_{\alpha}$, we see that each splitting of a clopen
interval into $r>1$ clopen intervals preserves an additional $r-1$
points of $X$. It follows that it is possible to construct $X_{\omega}$
from $X$, by means of the closed nowhere dense set construction,
which contains a dense-in-itself subset of cardinality $\ge\aleph_{0}$.
That is, it is possible to force $c-c\times\aleph_{0}\ne\slashed{O}$.
\\
\\
\includegraphics[scale=0.4]{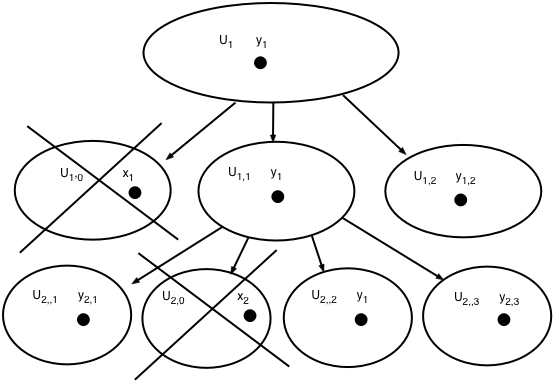}\\
\\
\emph{Figure 3: An example of how a descending sequence of clopen
intervals can be forced to contain one point of a set $X$ per node
of a decomposition of $X$ into clopen intervals.}\\
\\
We can state this as:
\begin{thm}
In a Baire topology of an uncountable set of real numbers, X, that
comprises a set of binary sequences with a countable clopen base and
no discrete or isolated points, it is always possible to construct
a Cantor set $X_{\omega}$ from $X$ which contains a dense-in-itself
subset of cardinality $\ge\aleph_{0}$. That is, it is possible to
force $c-c\times\aleph_{0}\ne\slashed{O}$. But deleting $\aleph_{1}$
clopen intervals in a dense way results in the empty set, i.e. $c-c\times\aleph_{1}=\slashed{O}$.
\end{thm}

\section{A natural topology of sets of higher order}

The Baire topology can be defined in the case of higher order sets
in the same way as real numbers with the difference that clopen intervals,
$u_{\alpha}(x)$, can split from a branch at any ordinal $\alpha<\aleph$
rather than $\alpha<\omega$. The decision to allow splits that occur
at nodes of infinite height has the consequence that the Baire topology
on $X$ is not equivalent to a product topology, which is in turn
equivalent to allowing only clopen sets that split from a branch of
height $n<\omega$ in the case of a finite base for the product\footnote{In this case for a product topology a finite sequence of bounded finite
sets (\emph{i.e.} an initial finite $n$-ary sequence for some natural
number $n$) will define the topology.}. \\
\\
If $X$ has a dense-in-itself kernel\footnote{If $X$ does not have a dense-in-itself kernel then $\left|X\right|\le\aleph$.},
then it is possible to construct closed nowhere dense sets by removing
clopen intervals in the same manner as the case of sets of real numbers,
deleting a $\aleph$-sequence $S=\langle x_{\beta<\aleph}\in X\rangle$
that is dense in $X$ by means of the construction $cntr(Y;x_{\beta}):=Y-u_{\beta}(x_{\beta})$,
where $u_{\beta}(x_{\beta})=\{y:(y)_{n(\beta)+r(\beta)+1}\ne(x_{\beta})_{n(\beta)+r(\beta)+1}\wedge(\forall\gamma\le n(\beta)+r(\beta))[(y)_{\gamma}=(x_{\beta})_{\gamma}]\}$.
$n(\beta)$ is the supremum of nodes where $x_{\beta}$ splits from
$x_{\delta<\alpha}$ and the offset $r(\beta)>0$ is any ordinal $r(\beta)<\aleph$
(as in the case of the real numbers, skipping over empty nodes). It
follows that we can construct a sequence $X_{0}=X$, $X_{\delta+1}=cntr(X_{\delta};x_{\delta})$
for $\delta<\aleph$ and $X_{\lambda}=\bigcap_{\delta<\lambda}X_{\delta}$
for limit ordinal $\lambda\le\aleph$. We claim that $X_{\aleph}$
is a closed nowhere dense set, which follows because the construction
results in sets of the form $X-\bigcup_{\beta<\aleph}u_{\alpha(\beta)+r(\beta)}(x_{\beta})$,
$i.e.$ the complement of an open set, and any clopen interval will
have a clopen interval deleted from it (since the set of sequences
$\{S:S\in X_{\alpha}\}$ is dense in $X_{\alpha}$). Finally, we note
that if $X$ has a linear rather than tree representation in terms
of binary $\aleph$-sequences, just as in the case of the real numbers
we can remove isolated points by (transfinite) induction, starting
at the initial member of the linear order, and proceeding until all
members of $X$ have become isolated. In this case $X$ has cardinality
$\aleph$.\\
\\
We can state this as:
\begin{thm}
In the Baire topology of a set of binary $\aleph$-sequences, X, with
a basis of clopen intervals of cardinality $\aleph$ and no discrete
or isolated points, the Cantor sets $X_{\aleph}$ constructed from
$X$ and a dense $\aleph$-sequence $\langle x_{\beta<\aleph}\in X\rangle$
by $X_{0}=X$, $X_{\delta+1}=cntr(X_{\delta};x_{\delta})$ for $\delta<\aleph$
and $X_{\lambda}=\bigcap_{\delta<\lambda}X_{\delta}$ for limit ordinal
$\lambda\le\aleph$ are closed and nowhere dense. 
\end{thm}

In the same way as in the case of the real numbers it is possible
to force $\left|X_{\aleph}\right|\ge\aleph$ by applying the closed
nowhere dense set construction to $X$, which has a dense-in-itself
subset and has cardinality $\aleph<c\le2^{\aleph}$ such that all
clopen sets have cardinality $c$ (removing all clopen intervals of
cardinality $<c$ if necessary), to construct an $X_{\aleph}$ which
contains any given $x\in X$ by choosing the set $\langle x_{\alpha<\aleph}\rangle$
of $\aleph$-sequences to be deleted to be such that $x_{\alpha<\aleph}\in X,$
$x_{\alpha}\ne x$ and $\langle x_{\alpha<\aleph}\rangle$ is dense
in $X$,\footnote{There are at least $\aleph$ such $\aleph$-sequences because if there
were $<\aleph$, then some clopen interval will be $\aleph-$sequentially
complete. } and by modifying $cntr$ to increase the value of $r(\beta)$ so
that $x\in X_{\alpha}$ for all $\alpha<\aleph$ and therefore $x\in X_{\aleph}$
by definition. \\
\\
If we consider that each clopen interval is divided into $r>1$ disjoint
clopen intervals and one clopen interval is deleted, we can write
$U_{\alpha}=\bigcup_{0\le\beta\le r(\alpha)}U_{\alpha,\beta}$ and
$U_{\alpha+1}=U_{\alpha,\beta}$ for any choice of $\beta$ such that
$1\le\beta\le r(\alpha)$, where $U_{\alpha,\beta}$ are clopen intervals,
$U_{\alpha,0}$ is deleted because $x_{\alpha}$ is a branch in $U_{\alpha,0}$,
$U_{1}=X$ and $U_{\lambda}=\bigcap_{\beta<\lambda}U_{\beta}$, for
ordinal numbers $\alpha,\:\beta,\:r(\alpha)<\aleph$ and $\lambda$
a limit ordinal. Then if $U_{\alpha}$ preserves $y_{\alpha}\in U_{\alpha}$,
\emph{i.e.} $y_{\alpha}\in U_{\aleph}$, then if $y_{\alpha}\in U_{\alpha,\beta}$
for some ordinal number $\beta>0$ there are $y_{\alpha,\gamma}\in U_{\alpha,\gamma}\ne y_{\alpha}$
for all $1\le\gamma\le r(\alpha)$ and $\gamma\ne\beta$. We require
that $U_{\alpha}$ is constructed to include a clopen set around branch
$x_{\alpha}$ in $U_{\alpha,0}$ (which will be deleted), to preserve
$\bigcup_{\gamma<\alpha}\{y_{\gamma}\}$ and $y_{\alpha}\in U_{\alpha,\beta}$
for some $\beta>0$. This requirement can be met by selecting $y_{1}\in U_{1}$
such that $y_{1}\ne x_{\alpha<\aleph}$ and by constructing $U_{\alpha+1}$
and $y_{\alpha+1}$ and $U_{\lambda}$ and $y_{\lambda}$ for limit
ordinal $\lambda<\aleph$ as follows given clopen interval $U_{\alpha}$
and $y_{\alpha}\in U_{\alpha}$ which is preserved, $i.e.$ $y_{\alpha}\in U_{\aleph}$,
or clopen intervals $U_{\beta<\lambda}$ and $y_{\beta}\in U_{\beta}$
in the case of limit ordinal $\lambda$. \\
\\
\noindent\fbox{\begin{minipage}[t]{1\columnwidth - 2\fboxsep - 2\fboxrule}%
\begin{itemize}
\item If $\alpha<\aleph$ is a successor ordinal:\\
\\
If $x_{\alpha}$ is a branch in $U_{\alpha}$: set $r(\alpha)$ to
include the node where $y_{\alpha}$ splits from $x_{\alpha}$, set
$U_{\alpha,,0}\subset U_{\alpha}$ to be a clopen interval such that
$x_{\alpha}$ is a branch in $U_{\alpha,0}$ and $y_{\alpha}\notin U_{\alpha,0}$,
and set $U_{\alpha+1}:=U_{\alpha,\beta}$ for any choice of $1\le\beta\le r(\alpha)$
where $U_{\alpha}=\bigcup_{0\le\beta\le r(\alpha)}U_{n,m}$. Set $y_{\alpha+1}:=y_{\alpha}$
if $y_{\alpha}\in U_{\alpha+1}$ and otherwise choose $y_{\alpha+1}\in U_{\alpha+1}$
such that $y_{\alpha+1}\ne y_{1\le\beta\le\alpha}$ (which is possible
as there are at least $c>\aleph$ members in any clopen interval such
as $U_{\alpha+1}$)\footnote{The construction also works for $c=\aleph$ as $\alpha+1<\aleph$}.\\
\\
If $x_{\alpha}$ is a not branch in $U_{\alpha}$: set $U_{\alpha+1}:=U_{\alpha}$
and set $y_{\alpha+1}:=y_{\alpha}$.\\
\item If $\alpha<\aleph$ is a limit ordinal:\\
\\
Set $U_{\alpha}:=\bigcap_{1\le\beta<\alpha}U_{\beta}$. \\
\\
By transfinite induction $U_{\alpha}$ preserves at least one $y\in U_{\alpha}$
such that $y\in U_{\beta}$ is preserved for all $\beta$ such that
$1\le\beta<\alpha$. But for every such $y$ there is a least ordinal
$\alpha\le\gamma<\aleph$ such that there are no deletions from any
subtrees that split from $y$ at or above the $\gamma$-th node of
$y$ (as the cardinality of the union of $<\aleph$ ordinals $<\aleph$
is $<\aleph$ using the Axiom of Choice). For ease of reference, the
least ordinal $\gamma$ is written as $h(\alpha,y)$. It follows that
$U_{\alpha}$ contains non-empty clopen intervals, $V_{\alpha}:=U_{h(\alpha,y)}$
for all $y\in U_{\alpha}$ . \\
\\
Set $y_{\alpha}:=y$ for any choice of $y\in X$ such that $y\in U_{\beta}$
for all $\beta<\alpha$. There is always at least one such $y$ because
$y\in V_{\alpha}\subseteq U_{\alpha}$. \\
\item If $\alpha=\aleph$:\\
\\
Set $U_{\alpha}:=\bigcap_{1\le\beta<\alpha}U_{\beta}$. Then each
set $U_{h(\aleph,y)}=\{y\}$ and $y_{\aleph}:=y$ for all $y\in U_{\aleph}$.
\\
\end{itemize}
\end{minipage}} \\
\\
 \\
Since each $y_{\alpha}$ can be preserved by the same construction
as was used for $y_{\beta<\alpha}$, we see that each splitting of
a clopen interval into $r>1$ clopen intervals for $r<\aleph$ preserves
an additional $r-1$ points of $X$, and all of these points are preserved
at limit ordinals (as represented by all possible values of $U_{\lambda}$
for limit ordinals $\lambda\le\aleph$). It follows that every $X_{\aleph}$
generated from $X$ by the closed nowhere dense set construction contains
a dense-in-itself subset of cardinality $\ge\aleph$. It follows that
if $\aleph<\left|X\right|\le2^{\aleph}$ then it is possible to force
$\left|X\right|-\aleph\times\left|X\right|\ne\emptyset$, while if
a dense sequence of $\aleph+1$ clopen intervals are deleted then
$\left|X\right|-(\aleph+1)\times\left|X\right|=\emptyset$. \\
\\
\includegraphics[scale=0.4]{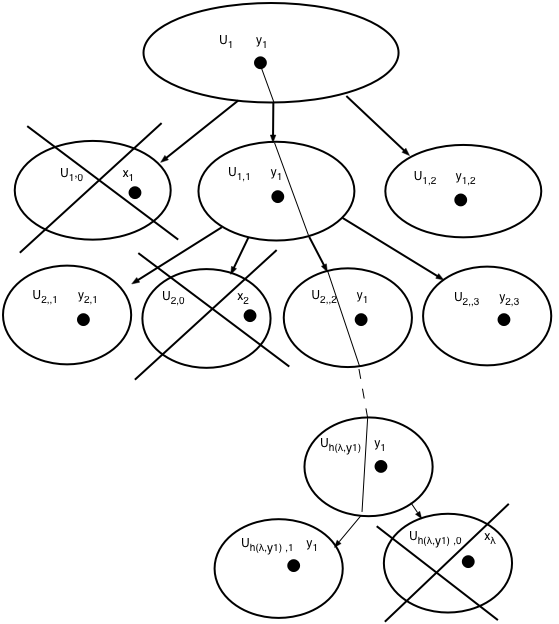}\\
\\
\emph{Figure 4: An example of how a descending sequence of clopen
intervals has clopen intervals from some limit ordinal onwards, at
the point where no clopen intervals have yet been deleted in the construction.}\\
\\
We can state this as:
\begin{thm}
\label{thm:Baire}In a Baire topology of a set of binary $\aleph$-sequences,
X, such that $\aleph<\left|X\right|\le2^{\aleph}$ with a basis of
clopen intervals of cardinality $\aleph$ and no discrete or isolated
points, it is always possible to construct a Cantor set $X_{\aleph}$
from $X$, which contains a dense-in-itself subset of cardinality
$\ge\aleph$. That is, if $\aleph<\left|X\right|\le2^{\aleph}$ then
it is possible to force $\left|X\right|-\aleph\times\left|X\right|\ne\emptyset$,
while if a dense sequence of $\aleph+1$ clopen intervals are deleted
then $\left|X\right|-(\aleph+1)\times\left|X\right|=\emptyset$. 
\end{thm}

If $X$ is $\aleph$-sequentially complete and therefore has a base
of clopen intervals which are $\aleph$-sequentially complete, \emph{i.e.}
all paths of length $\aleph$ through the interval are members of
the interval, then we can claim that it is possible to force $2^{\aleph}-\aleph\times2^{\aleph}=2^{\aleph}$
because the same labelling technique can be used on clopen intervals
as in the case of the real numbers (all deleted clopen intervals being
labelled 0) and we can note that all $\aleph$-sequences of ordinal
labels $\aleph>\alpha>0$ are members of $X$ by $\aleph$-sequential
completeness and that the cardinality of $\aleph^{\aleph}=2^{\aleph}$.
By transfinite induction for $\alpha<\aleph$ with the hypothesis
that all clopen intervals $\subseteq X_{\alpha}$ have cardinality
$2^{\aleph}$, at stage $\alpha+1$ $X_{\alpha}$ will be split into
$>1$ and $<\aleph$ clopen intervals with a label $\ne$0, each of
which by the induction hypothesis has cardinality $2^{\aleph}$ ,
so $X_{\alpha+1}$ as the union of these sets, will also have cardinality
$2^{\aleph}.$ For a limit ordinal $\lambda$, all clopen intervals
with label $0$ can be deleted, and for the clopen intervals remaining
$\aleph$-sequential completeness can be applied to the paths between
labels formed at stages successor stages $\alpha<\lambda$ to show
that $X_{\lambda}$ has cardinality $2^{\aleph}$. The latter observation
relies on the fact that a strictly descending $\aleph$-sequence of
non-empty clopen intervals defines a single point or branch $x\in X$,
and therefore a descending $\alpha<\aleph$-sequence of clopen intervals
can be identified with an initial segment of $x$ of length $\alpha$.\\
\\
We can state this as:
\begin{thm}
In a $\aleph$-sequentially complete Baire topology of a set of binary
$\aleph$-sequences, X, with a basis of clopen intervals of cardinality
$\aleph$ and no discrete or isolated points, the Cantor sets have
cardinality $2^{\aleph}$ and the process of deleting $\aleph$ clopen
intervals gives rise to the equation $2^{\aleph}-\aleph\times2^{\aleph}=2^{\aleph}$.
\end{thm}

The Baire Category Theorem can also be generalized to the statement
that in a $\aleph$-sequentially complete Hausdorff topological space,
$X$, that comprises a $\aleph$-sequentially complete set of binary
$\aleph$-sequences with a clopen base of cardinality $\aleph$ and
with no discrete or isolated points, $X$ is not the union of $<\aleph+1$-many
nowhere dense subsets for $\aleph\ge\aleph_{0}$.\footnote{See for example \cite{key-4} Proposition 3.8 p. 213 for the case
$\aleph=\aleph_{0}$.} It is worth noting that a $\aleph$-sequentially complete Hausdorff
space $X$ that comprises a $\aleph$-sequentially complete set of
binary $\aleph$-sequences with  a clopen base of cardinality $\aleph$
is neither compact \footnote{\label{fn:Construct-a-cover}Construct a cover $Z$ of $X$ as follows.
Fix a branch $x\in X$ and add to $Z$ all disjoint clopen intervals
that split from $x$. Then add to $Z$ a clopen interval that splits
from $y\in X$ such that $y\ne x$ at a node of index $>\aleph_{0}$
. $Z$ has no finite open subcover if the base of $X$ has cardinality
$\aleph>\aleph_{0}$ because there are at least $\aleph_{0}$ disjoint
clopen intervals in the cover such that removal of any one such set
would not result in a cover of $X$. A corollary is that there is
a descending $\aleph$-sequence of clopen sets $\langle x_{\alpha<\aleph}\rangle$
(complements of clopen intervals) such that all finite intersections
of $X_{\alpha}$ are non-empty while $\bigcap_{\beta<\aleph}X_{\beta}=\emptyset$.
The failure of compactness means that the topological space cannot
be characterized by convergent ultrafilters, but it possible nevertheless
to characterize $X$ by the set of strictly descending $\aleph$-sequences
of clopen intervals converging to a point $x$, and in fact a generalized
local compactness condition does hold for any $\aleph$-sequentially
complete Hausdorff topological space that has a clopen base of cardinality
$\aleph$: if for every $\beta<\aleph$ $\bigcap_{\alpha<\beta}F_{\alpha}\ne\emptyset$
then $\bigcap_{\alpha<\aleph}F_{\alpha}\ne\emptyset$ for any strictly
descending $\aleph-$sequence of non-empty clopen intervals $F_{\alpha}$,
\emph{i.e.} $F_{\beta}\subset F_{\gamma}$ if ordinal $\gamma<\beta$.
This follows by following the branch from which successive nested
clopen intervals split.} nor metrizable \footnote{\label{fn:By-the-Nagana-Smirnov}By the Nagata-Smirnov metrization
theorem $2^{\aleph}$ for $\aleph>\aleph_{0}$ is not metrizable as
it is Hausdorff and regular (since any two points can be separated
by clopen neighbourhood), but does not have a countable locally finite
base (since there are uncountably many clopen intervals and if every
member of $2^{\aleph}$ is only a member of finitely many clopen intervals,
the family of clopen intervals in the base is uncountable). } for $\aleph>\aleph_{0}$, but there is a generalized metric function
that can be used.\\
\\
In \cite{key-2} R. Kopperman showed that it possible to replace the
set of real numbers in the definition of a metric space with a commutative
semi-group\footnote{A semi-group is defined like a group but may lack an inverse operation
to the group operation.}, and for every topology to find a suitable commutative semi-group
for which a metric can be introduced to the topological space (which
may not be symmetric or separate distinct members of the topological
space). Let $\langle2^{\aleph},\oplus\rangle$ be a structure defined
as follows. If $2^{\aleph}$ is the set of functions $\aleph\rightarrow2$
and $a,\,b\in2^{\aleph}$, \emph{i.e.} are binary $\aleph$-sequences,
then treat $a$ and $b$ as $\aleph$-sequences of real numbers in
the range $[0,\infty)$, $\langle a_{1},...,a_{\alpha<\alpha},\ldots\rangle$
and $\langle b_{1},...,b_{\alpha<\alpha},\ldots\rangle$ for real
numbers $a_{\alpha},b_{\alpha}\in[0,\infty)$, and define $a\oplus b$
as the $\aleph$-sequence $\langle a_{1}+b_{1},...,a_{\alpha<\alpha}+b_{\alpha<\alpha},\ldots\rangle$.
\\
\\
Let us denote a clopen interval comprising binary $\beth$-sequences
from $a\le b$ to $b$ by $([a,b])[[\beth]$ and the half-open interval
from $a<b$ to $<b$ by $[a,b)[\aleph]$. Let us now define $d(x,y)$
for binary $\aleph$-sequences of the form $\langle x,...,x_{\alpha<\alpha},\ldots\rangle$
where real number $x_{\alpha}\in([0,1])[\omega]$, by $d(x,x):=0$
and $d(x,y):=1_{\alpha(x,y)}$, \emph{i.e.} where there is a 1 only
in the $\alpha$-th digit of an $\aleph$-sequence of binary $\omega$-sequences
with a binary point (a real number) and 0 for all other digits, and
\emph{$\alpha$} is a successor ordinal that is the height of the
lowest node where $(x)_{\alpha}\ne(y)_{\alpha}$. This is an unambiguous
definition because each $x_{\alpha}\in([0,1])$ can be represented
as a real number with 0 in front of the binary point (because $1.000\ldots$
can also be written $0.111\ldots).$ We can thus skip the $0.$ before
the binary point in the real number representation uniquely identifying
the height of the lowest node where $(x)_{\alpha}\ne(y)_{\alpha}$.
In practice we will leave the binary point in place for clarity. Surprisingly
we have $d(x,y)\le\frac{1}{2}=\langle0.1,0,0,0,\ldots\rangle$ because
the first node after $0.$ is the first node that $x$ and $y$ can
differ. On the other hand $d(x,y)\oplus d(y,z)\le1$, and a sum of
natural number $n$ such distances is bounded by $n/2$. \\
\\
We can show that if $x$ and $y$ are binary $\aleph$-sequences in
the clopen interval $([0,1])[\aleph]$ then $\langle([0,1])[\aleph],d\rangle$
forms a metric space.\footnote{In fact $([0,1])[\aleph]$ is also an ultrametric space as $max(d(x,y),d(y,z))\ge d(x,z)$,
see Figure 5. } We have $d(x,y)=d(y,x)$, $d(x,y)\ge0$ and $d(x,y)=0\rightarrow x=y$
immediately from the definition of $d$ and the fact that all real
numbers in $x$ and $y$ start with $0.$ We also have $d(x,y)\oplus d(y,z)\ge d(x,z)$
because:
\begin{enumerate}
\item If $\alpha(x,z)>\alpha(x,y)$ then $\alpha(y,z)=\alpha(x,y)$, and
$d(x,y)\oplus d(y,z)=1_{\alpha(x,y)}+1_{\alpha(y,z)}=1_{\alpha(x,y)-1}>1_{\alpha(x,y)}>1_{\alpha(x,z)}$.
\item If $\alpha(x,z)<\alpha(x,y)$ then $\alpha(y,z)=\alpha(x,z)$, and
$d(x,y)\oplus d(y,z)=1_{\alpha(x,y)}+1_{\alpha(y,z)}>1_{\alpha(x,z)}$. 
\item If $\alpha(x,z)=\alpha(x,y)$ then we have $d(x,y)\oplus d(y,z)=1_{\alpha(x,y)}+1_{\alpha(y,z)}>1_{\alpha(x,z)}$. 
\item If $x=y$ then $\alpha(x,z)=\alpha(y,z)$, $d(x,y)=0$ and $d(x,y)\oplus d(y,z)=0+1_{\alpha(y,z)}=1_{\alpha(x,z)}$;
if $y=z$ then $\alpha(x,y)=\alpha(x,z)$, $d(y,z)=0$ and $d(x,y)\oplus d(y,z)=1_{\alpha(x,y)}+0=1_{\alpha(x,z)}$;
and if $x=z$ then $\alpha(x,y)=\alpha(y,z)$, $d(x,z)=0$ and $d(x,y)\oplus d(y,z)=1_{\alpha(x,y)}+1_{\alpha(y,z)}>0$. 
\end{enumerate}
We can define clopen intervals in the Baire topology as $\{y:d(x,y)=1_{\alpha(x,y)}\}$.
\\
\\
\includegraphics[scale=0.7]{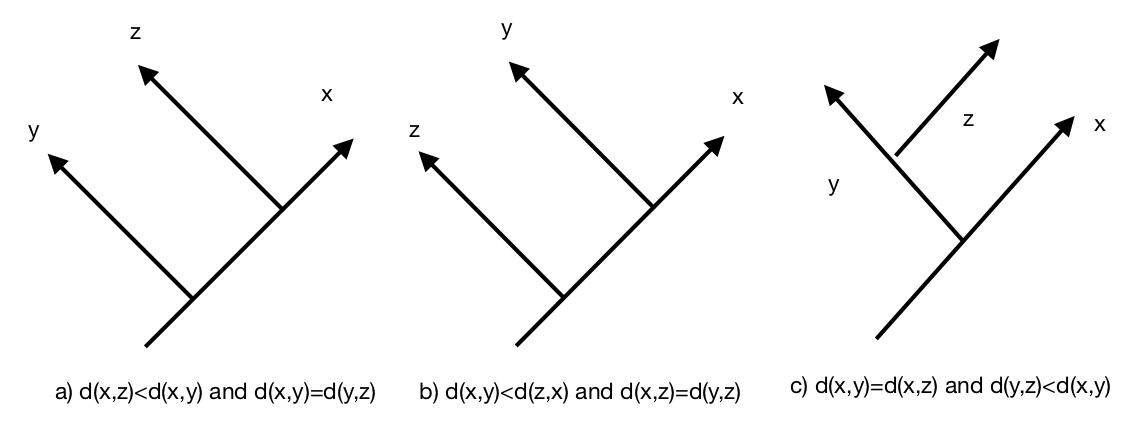}\\
\\
\emph{Figure 5: Diagrams showing the different cases in the generalized
metric of $([0,1])[\aleph]$.}\\
\\
The clopen interval $([0,1])[\aleph]$ was chosen for simplicity,
and it has the closure property $d(x,y)\oplus d(y,z)\in([0,1])[\aleph]$
if $x,y,z\in([0,1])[\aleph]$; but exactly the same generalized metric
works on the interval $[0,\infty)[\aleph]$. Any binary real number
can be padded with 0s in front of the binary point if necessary to
have a prefix of the same length as any other binary real number,
and all binary digits in the prefix are treated as negative whole
number offsets from the binary points. For example, to calculate the
$d(x,y)$ where $x=11.000\ldots$ and $y=100.000\ldots$ the prefix
of $x$ can be padded to $011$ and $d(x,y)=100.000\ldots$, which
is at position -3 with respect to the binary point. It is true that
$d(x,y)\oplus d(y,z)\in[0,\infty)[\aleph]$ if $x,y,z\in[0,\infty)[\aleph]$,
but of course $[0,\infty)[\aleph]$ is not closed under upward limits,
$i.e.$ $d(x,y)\rightarrow\infty$ if $x$ is fixed and $y\rightarrow\infty$
or \emph{vice versa}, and $\infty\notin[0,\infty)$. The clopen interval
$([0,1]))[\aleph]$ is therefore a better representation of the set
of all binary $\aleph$-sequences. \\
\\
 We should be clear that for $\aleph>\aleph_{0}$ the generalized
metric space is not compact. The reason is that, as we have seen,
it is possible to have an $\omega$-sequence of non-empty clopen and
totally bounded intervals $\langle X_{\alpha<\omega}\rangle$ such
that $X_{\alpha}\subseteq X_{\beta}$ if $\beta\leq\alpha<\omega$
and $\bigcap_{\beta<\omega}X_{\beta}=\emptyset$ . But the following
statements are true in a generalized metric space. If $y$ is a limit
point of non-empty $\bigcap_{\alpha<\aleph}X_{\beta}$ where $\langle X_{\alpha<\aleph}\rangle$
is an $\aleph$-sequence of non-empty clopen intervals such that $X_{\beta}\subset X_{\gamma}$
if $\gamma<\beta$, then $y\in\bigcap_{\alpha<\aleph}X_{\beta}$.\footnote{If $\bigcap_{\beta<\aleph}X_{\beta}\ne\emptyset$ and $y\in\overline{\bigcap_{\beta<\aleph}X_{\beta}}-\bigcap_{\beta<\aleph}X_{\beta}$
then $y\notin X_{\alpha}$ for some $\alpha<\aleph$, and since $X_{\alpha}$
is clopen, $y\notin\overline{X_{\alpha}}$ and hence $y\notin\bigcap_{\beta<\aleph}\overline{X_{\beta}}$.
Since $\overline{\bigcap_{\beta<\aleph}X_{\beta}}\subseteq\bigcap_{\beta<\aleph}\overline{X_{\beta}}$
by definition of closure, it follows that $y\notin\overline{\bigcap_{\beta<\aleph}X_{\beta}}$,
contradiction. } Moreover, as noted in Footnote \ref{fn:Construct-a-cover} and which
can be seen from from the proof of the generalized Baire Category
Theorem below, in a $\aleph$-sequentially complete Hausdorff space
every strictly descending nested $\aleph$-sequence of non-empty clopen
intervals converges to exactly one point. Furthermore, the compactness
condition can be replaced by a generalized compactness condition in
a $\aleph$-sequentially complete Hausdorff topological space called
\emph{$<\aleph$-compactness:} if for every $\beta<\aleph$ $\bigcap_{\alpha<\beta}X_{\alpha}\ne\emptyset$
then $\bigcap_{\alpha<\aleph}X_{\alpha}\ne\emptyset$ for any strictly
descending $\aleph$-sequence of non-empty clopen intervals $\langle X_{\alpha<\aleph}\rangle$,
\emph{i.e.} $X_{\beta}\subset X_{\gamma}$ if ordinal $\gamma<\beta$.
It is therefore true that a $\aleph$-sequentially complete $\langle2^{\aleph},\oplus\rangle$-generalized
metric space is a $<\aleph$-compact topological space (\emph{i.e}.
a topological space such that each closed set satisfies the $<\aleph$-compactness
condition). \\
\\
The proof of the generalized Baire Category Theorem proceeds as follows
(broadly following \cite{key-4} 3.83Ac 213 for the case $\aleph=\aleph_{0}$).
Let us suppose for contradiction that $X=\bigcup_{\alpha<\aleph}C_{\alpha}$
for closed nowhere dense sets $C_{\alpha}$. We claim we can find
a $\aleph$-sequence of descending non-empty closed sets $\langle D_{\alpha}:\alpha<\aleph\rangle$
such that $D_{0}\subseteq X$, $D_{\beta}\subseteq D_{\alpha}$ if
$\alpha<\beta$ and $C_{\alpha}\cap D_{\alpha+1}=\slashed{O}$. This
is possible because for every non-empty open set $O$, $O-C_{\alpha}$
is a non-empty open set as $O$ has a non-empty interior and $C_{\alpha}$
has an empty interior. We choose $D_{0}\subseteq X$ to be a clopen
interval (since the space has a clopen base), $D_{\alpha+1}\subseteq D_{\alpha}-C_{\alpha}$
to be a clopen interval (as a clopen subset of the non-empty interior
of $D_{\alpha}$) and $D_{\lambda}:=\bigcap_{\alpha<\lambda}D_{\alpha}$
for limit ordinals $\lambda$. We can see that $D_{\alpha<\aleph}\ne\slashed{O}$
because at limit ordinals, $\lambda$, the node from which the clopen
interval $D_{\lambda}$ splits from some branch $x\in X$\footnote{In fact the $\aleph-$sequence of initial segments from which nested
clopen intervals split defines a branch that is in $D_{\aleph}$.} has an ordinal which is the limit of an $\alpha$-sequence for $\alpha<\aleph$
of ordinals $<\aleph$ (because the branches are of length $\aleph$),
and thus the node has an ordinal $<\aleph$ (by the Axiom of Choice).
As $D_{\alpha<\aleph}$ can be viewed as a clopen interval splitting
from a branch, and that clopen interval is then split at some branch
in the interval at higher ordinals, it follows by $\aleph$-sequential
completeness that $D_{\aleph}$ can be identified with a set containing
an $\aleph$-sequence of $2^{\aleph}$, \emph{i.e.} a set containing
a single point. Using this observation we have $\bigcap_{\alpha<\aleph}D_{\alpha}=\{x\}$
for $x\in2^{\aleph}$, and since $D_{0}\subseteq X$, $x\in X.$ However,
as $C_{\alpha}\cap D_{\alpha+1}=\slashed{O}$, $x\notin\bigcup_{\alpha<\aleph}C_{\alpha}$.
Hence $X\ne\bigcup_{\alpha<\aleph}C_{\alpha}$, as was to be proved.\\
\\
We can state these results as:
\begin{thm}
(Generalized Baire Category Theorem) In a $\aleph$-sequentially complete
Hausdorff topological space, X, that comprises a $\aleph$-sequentially
complete set of binary $\aleph$-sequences with a clopen base of cardinality
$\aleph$ and with no discrete or isolated points, X is not the union
of $<\aleph+1$-many nowhere dense subsets for $\aleph\ge\aleph_{0}$. 
\end{thm}

\begin{thm}
\label{thm:A--sequentially-complete}A $\aleph$-sequentially complete
Hausdorff topological space that comprises a $\aleph$-sequentially
complete set of binary $\aleph$-sequences is not compact and not
metrizable for $\aleph>\aleph_{0}$, but it is possible to use a generalized
metric and every strictly descending $\aleph$-sequence of non-empty
clopen intervals converges to exactly one point.
\end{thm}

\section{A Modal Model of Set Theory}

We have seen from Theorem 10 that in a $\aleph$-sequentially complete
Hausdorff topological space, $X$, that has a clopen base of cardinality
$\aleph$ and with no discrete or isolated points, \emph{X} is not
the union of $<\aleph+1$-many nowhere dense subsets for $\aleph\ge\aleph_{0}$.
But is it the case that $X$ is the union of $\aleph+1$ closed nowhere
dense sets if the cardinality of $X>\aleph$? The answer is that this
result is possible because it can be forced if the $\aleph+1$ closed
nowhere dense sets are dense in $X$, but the forcing is quite natural.
\cite{key-3} provides a clear explanation of set theoretic forcing.
It will be seen that the result is independent of Zermelo Fraenkel
set theory with the Axiom of Choice (ZFC). The result can be seen
by means of the following construction.\\
\\
If we represent members of a $\aleph$-sequentially complete Hausdorff
topological space, $X$, as $<\aleph+1$-sequences, we can define
$X(\langle x_{1},\ldots,x_{\alpha<\aleph+1}\rangle;\langle y_{1},\ldots,y_{\beta<\aleph+1}\rangle)$
as the generalized Cantor (\emph{i.e.} closed nowhere dense) set that
results from the construction in Theorem 8 that preserves members
of $\langle x_{1},\ldots,x_{\alpha<\aleph+1}\rangle$ and deletes
members of $\langle y_{1},\ldots,y_{\beta<\aleph+1}\rangle$, where
each $x_{\gamma\le\alpha},\:y_{\gamma\le\beta}\in X$ and $x_{\gamma}\ne y_{\delta}$
for all $\gamma\le\alpha,\:\delta\le\beta$. Note that the choice
of $x_{\gamma}$ depends on $y_{\delta\le\gamma}$. Consider a $\aleph+1$-sequence
$X_{\alpha<\aleph+1}(\langle x_{1},\ldots,x_{\alpha}\rangle;\langle y_{1},\ldots,y_{\alpha}\rangle)$
of $\aleph$-sequentially complete closed nowhere dense sets, which
is possible because it is always possible to cover a branch of $\aleph$
nodes with disjoint sets of branches with $\aleph$ members. Then
we can see that:\\
\\
\noindent\fbox{\begin{minipage}[t]{1\columnwidth - 2\fboxsep - 2\fboxrule}%
$\bigcup_{\alpha<\aleph+1}X_{\alpha}(\langle x_{1},\ldots,x_{\alpha}\rangle;\langle y_{1},\ldots,y_{\aleph}\rangle)\cup\bigcup_{\alpha<\aleph+1}X_{\alpha<\aleph+1}(\langle y_{1},\ldots,y_{\alpha}\rangle;\langle x_{1},\ldots,x_{\aleph}\rangle)$%
\end{minipage}} \textcompwordmark\\
is dense in $X$ (because $\langle x_{1},\ldots,x_{\aleph}\rangle\cup\langle y_{1},\ldots,y_{\aleph}\rangle$
is dense in $X$) and it is possible for $\langle x_{1},\ldots,x_{<\aleph+1}\rangle$
and $\langle y_{1},\ldots,y_{<\aleph+1}\rangle$ to each have $\aleph+1$
members if the cardinality of $X>\aleph$. While it is not in general
true that the union of $\aleph+1$ closed nowhere dense sets, $X_{\alpha<\aleph+1}$,
that are dense in $X$ is $X$ (because a $\aleph$-sequence may exist
which is covered by $\aleph$-sequences that is in $X$ but is not
in the union), it is true in a natural model of $X.$ That model is
a transitive outer model model of cardinality $\aleph+1$ (see for
example \cite{key-3} in the case of countable transitive outer models),
in which the forcing partially ordered functions are $f:\aleph+1\rightarrow\{Y:Y\subseteq X\}$,
where $f_{\beta}=X_{\beta}$, where as above each $X_{\beta}$ is
a function of $\langle x_{1},\ldots,x_{\beta}\rangle$, $\langle y_{1},\ldots,y_{\beta}\rangle$
and of the function $r:\beta\rightarrow\beta$ used to control the
preservation of $\langle x_{1},\ldots,x_{\alpha}\rangle$ and the
deletion of $\langle y_{1},\ldots,y_{\alpha}\rangle$. Now in the
following keep $r$ fixed. To see that $f$ defines a partial ordering,
note that $f_{\gamma}\subseteq f_{\alpha}$ for $\aleph+1>\alpha>\gamma$
since $X_{\gamma}\subseteq X_{\alpha}$ for $\langle x_{1},\ldots,x_{\alpha}\rangle$
extending $\langle x_{1},\ldots,x_{\gamma}\rangle$\footnote{The $\aleph+1$-sequence of sets $\langle X_{\alpha<\aleph+1}\rangle$
must be eventually constant $\subset X$ in $<\aleph+1$ steps; otherwise
the preservation of $\aleph+1$ members of $X$ will result in X.}.\\
\linebreak{}
$F=\bigcup_{\beta<\aleph+1}f_{\beta}$ is a function because if: 
\begin{align*}
F(\langle x_{1},\ldots,,x_{<\aleph+1}\rangle;\langle y_{1},\ldots,y_{\aleph}\rangle)\ne F(\langle w_{1},\ldots,w_{<\aleph+1} & \rangle;\langle z_{1},\ldots,z_{\aleph}\rangle)
\end{align*}
then it follows that: 
\begin{align*}
f_{\beta}(\langle x_{1},\ldots,x_{\gamma<\beta}\rangle;\langle y_{1},\ldots,y_{\aleph}\rangle)\ne f_{\beta}(\langle w,\ldots,w_{\gamma<\beta}\rangle;\langle z_{1},\ldots,z_{\aleph}\rangle)
\end{align*}
for some $\beta<\aleph+1$ by definition of union, and there is a
correspondence (possibly many to one) between $\langle x_{1},\ldots,x_{\beta}\rangle$
and $X_{\beta}$. This implies that:
\begin{align*}
\langle x_{1},\ldots,x_{\gamma<\beta};y_{1},\ldots,y_{\aleph}\rangle\ne\langle w_{1},\ldots,w_{\gamma<\beta};z_{1},\ldots,z_{\aleph}\rangle
\end{align*}
since $f_{\beta}$ is a function and hence: 
\begin{align*}
\langle x_{1},\ldots,x_{\gamma<\aleph+1};y_{1},\ldots,y_{\aleph}\rangle\ne\langle w_{1},\ldots,w_{\gamma<\aleph+1};z_{1},\ldots,z_{\aleph}\rangle.
\end{align*}
\\
$F$ is onto $X-\bigcup_{\beta<\aleph}\{y_{\beta}\}$ because if $x\ne y_{\beta<\aleph}$
and $x\notin ran(F)$ for $x\in X$ then for some $\gamma<\aleph+1$
we can add $x$ to be preserved by $X_{\gamma}$ and all $X_{\alpha>\gamma}$
for $\alpha<\aleph+1$. Since the same argument works for $\langle y_{1},\ldots,y_{\aleph}\rangle;\langle x_{1},\ldots,x_{\alpha<\aleph+1}\rangle$
with $G=\bigcup_{\beta<\aleph+1}g_{\alpha}$, showing $G$ is onto
$X-\bigcup_{\beta<\aleph}\{x_{\alpha}\}$, we see that $F\cup G$
is a function onto $X$. This model is natural because it is completely
described by binary $\aleph+1$-sequences that can be instantiated
and that control membership of the closed nowhere dense sets. \\
\\
To see that this result is independent of ZFC, we note that the function
$F\cup G$ is a function from a set of cardinality $\aleph+1$ onto
a set of cardinality $2^{\aleph}$ (\emph{i.e. }from\emph{ }$\aleph+1$
onto $X$). Hence $\aleph+1\ge2^{\aleph}.$ Since $2^{\aleph}\ge\aleph+1$
by Cantor's theorem, GCH follows. Conversely if GCH is true, any union
of closed nowhere dense sets, such as $\{x\}$ for $x\in X,$ has
cardinality $\aleph+1=2^{\aleph}$ , and hence $X$ is the union of
$\aleph+1$ closed nowhere dense sets.\\
\\
We may state this result as:
\begin{thm}
(Not provable in ZFC, equivalent to GCH) In a $\aleph$-sequentially
complete Hausdorff topological space, X, that comprises a $\aleph$-sequentially
complete set of binary $\aleph$-sequences with a clopen base of cardinality
$\aleph$ and with no discrete or isolated points, X is the union
of $\aleph+1$-many nowhere dense subsets for $\aleph\ge\aleph_{0}$. 
\end{thm}

The construction showing that $X$ is the union of $\le\aleph+1$
closed nowhere dense sets is naturally constructed in $V_{\alpha}$
for some $\alpha>o(\aleph_{0})=\omega$, and uses the following argument:
if a counterexample could be produced, the construction could be applied
to the counterexample, showing that the counterexample would not be
an actual counterexample. It natural to think of these constructions
taking place in a modal model of ZFC (such as the S4 modal model of
\cite{key-7}). As a reminder, an S4-modal model of set theory is
a 4-tuple $\langle G,R,D^{G},F\rangle$, where $G$ is a set of forcing
conditions, $R$ is a reflexive and transitive relation on $G$, $D^{G}$
is the domain of sets corresponding to $G$ and $F$ is a mapping
from forcing conditions to quantifier-free sentences in set theory
with constants in $D^{G}$ such that $V$ can be extended to all sentences
in set theory with $D^{G}$ by means of the forcing relation $\Vdash$.
We have $p\Vdash A$ if $A\in F(p)$, $p\Vdash\neg X$ if $p\nVdash X$,
$p\Vdash X\wedge Y$ if $p\Vdash X$ and $p\Vdash Y$, $p\Vdash X\vee Y$
if $p\Vdash X$ or $p\Vdash Y$, $p\Vdash(\exists x)P(x)$ if $p\Vdash P(d)$
for some $d\in D^{G}$, $p\Vdash(\forall x)P(x)$ if $p\Vdash P(d)$
for all $d\in D^{G}$, $p\Vdash\Square X$ if $q\Vdash X$ for every
$q\in G$ such that $R(p,q)$, and $p\Vdash\lozenge X$ if $q\Vdash X$
for some $q\in G$ such that $R(p,q)$. In a modal model a sentence
of set theory $X$ is translated to a sentence of modal set theory
written $\llbracket X\rrbracket$, by induction: $\llbracket A\rrbracket=\Square\lozenge A$
for atomic A, $\llbracket\neg X\rrbracket=\Square\lozenge\neg\llbracket X\rrbracket$,
$\llbracket X\wedge Y\rrbracket=\Square\lozenge(\llbracket X\rrbracket\wedge\llbracket X\rrbracket$,
$\llbracket X\vee Y\rrbracket=\Square\lozenge(\llbracket X\rrbracket\wedge\llbracket X\rrbracket$,
$\llbracket(\exists x)P(x)\rrbracket=\Square\lozenge((\exists x)\llbracket P(x)\rrbracket$,
and $\llbracket(\forall x)P(x)\rrbracket=\Square\lozenge((\forall x)\llbracket P(x)\rrbracket$,
and it is proven that the translation of every instance of an axiom
of $ZFC$ is true for each forcing condition of the model. The model
that we have constructed is then $\langle\{x_{\alpha},y_{\alpha},X_{\alpha}:\alpha<\aleph+1\},\subseteq,2^{\aleph},F:\{x_{\alpha}y_{\alpha},X_{\alpha}\}\rightarrow\{x_{\alpha\le\aleph}\in X_{\beta\ge\alpha},y_{\alpha}\notin X_{\beta\ge\alpha}\}\rangle$,
where $\{x_{\alpha}y_{\alpha},X_{\alpha}\}$ are as described in Theorem
\ref{thm:Baire}. 

\section{A Generalized Metric Model of Set Theory}

We can also use the fact (see Theorem \ref{thm:A--sequentially-complete})
that any initial segment of $V$, $V_{\alpha}=2^{\aleph}$ for some
cardinal $\aleph$, can be considered as a $\langle2^{\aleph},\oplus\rangle$-generalized
metric space with the Baire topology on any set $X\subseteq2^{\aleph}$
comprising binary $\aleph$-sequences (or equivalently $\aleph$-sequences
of real numbers). That is to say, that for infinite $\alpha$ and
$\aleph$ $V_{\alpha}$ can be represented as a clopen interval $([0,1])[\aleph]$
for binary $\aleph$-sequences for some length $\aleph$. As we have
seen, if all real numbers in the $\aleph$-sequence start with the
same number ($0.$ in the case of $([0,1])[\aleph]$) then all binary
$\aleph$-sequences can still be represented (by ignoring the constant
number before the binary point). It is therefore reasonable to represent
$V_{\alpha}$ as a clopen interval $([0,1])[\aleph]$ for binary $\aleph$-sequences
for some length $\aleph$. That $V_{\alpha}$ is a generalized metric
space does not alter what sets exist, as those sets will be sets of
binary $\aleph$-sequences (for example most will not be $\aleph$-sequentially
complete); the constraint of being a generalized metric space only
determines how far apart points in the space are.\\
\\
It is then possible to decide the membership of $X$ in $<\aleph+1$
steps by enumeration as follows. Consider a clopen interval, $([0,1])[\aleph]$,
which is linearly ordered lexicographically, \emph{i.e.} $z<y$ if
$(\exists\alpha<\aleph)[(z_{\alpha}<y_{\alpha})\wedge(\forall\beta<\alpha)(z_{\beta}=y_{\beta})$
for $w_{\alpha}$ the $\alpha$-th binary member of the $\aleph$-sequence
$w$, and assume that each binary $\aleph$-sequence $z$ in $2^{\aleph}$
is marked with 1 or 0 depending whether $z\in X$ or not, which is
decidable only if you find the location of $z$ in the interval. The
latter assumption reflects the fact that when you search for $x$
in a linearly ordered set it is either present in its place in the
order (when $x\in X$) or it is not (when $x\notin X)$. Before we
begin the construction, we will need the ability to divide a binary
$\aleph$-sequence (of a $\aleph$-sequence of real numbers) by 2.
This is just standard binary division by 2 with carries to the right
if necessary.\\
\\
To start the construction, bisect the interval to give a point $m=\langle0.1,0,0,0,\ldots\rangle$.
Now set $r:=m$. If the midpoint $r=x$ then we can decide whether
$x\in X$ or $x\notin X$ and stop. Otherwise test whether $x<r.$
If $x<r$ then consider the clopen interval $([0,r])[\aleph]$; and
if $x>r$ consider the clopen interval $([r,1])[\aleph]$. Iterate
the bisection construction as follows\footnote{The clopen intervals are not subsets of $X$ in general but are subsets
of $2^{\aleph}$.}. $cl_{1}=([0,1])[\aleph]$, $cl_{\alpha+1}=Bi(cl_{\alpha};x)$ and
$cl_{\lambda}=\bigcap_{\alpha<\lambda}cl_{\alpha}$ for limit ordinal
$\lambda$ (which is the unique maximal clopen interval $\subseteq([0,1])[\aleph]$
such that for all $z\in cl_{\lambda}$ the initial $\lambda$-sequence
of $z$ is $x[\lambda]:=\langle x_{\alpha}:\alpha<\lambda\rangle$),
\emph{i.e}. $([x[\lambda]\parallel\langle0,0,0,\ldots\rangle,x[\lambda]\parallel\langle1,1,1,\ldots\rangle])$),
where $\parallel$ is concatenation, $\langle0,0,0,\ldots\rangle$
and $\langle1,1,1,\ldots\rangle$ are $\aleph$-sequences that stand
for $\aleph$ concatenated $\omega$-sequences $\langle0.0,0,0,\ldots\rangle$
and $\langle0.1,1,1,\ldots\rangle$ respectively, $Bi(([a,b])[\aleph];x)=([a,r])[\aleph]$
and $x_{\alpha}=0$ if $([a,b])[\aleph]=cl_{\alpha}$ and $x<r$ for
the midpoint $r=(b-a)/2$, $Bi(([a,b])[\aleph];x)=([r,b])[\aleph]$
and $x_{\alpha}=1$ if $([a,b])[\aleph]=cl_{\alpha}$ and $x>r$,
and the iteration stops if $x=r$ (and one can decide whether $r\in X)$.
It is clear that the construction will terminate in $\le\aleph$ steps,
as a nested sequence of clopen intervals can only comprise $\aleph$
members, as that is how many bits there are in the single binary $\aleph$-sequence
in any non-empty intersection of a nested sequence of clopen intervals.
If $x\in X$ has not been confirmed in $<\aleph$ steps, then at the
$\aleph$ step $cl_{\aleph}=([x,x])=\{x\}$, and at ordinal step $\aleph+1$
(\emph{i.e.} 1 after $\aleph)$ we can then decide whether $x\in X$
given that $x$ has been located.\\
\\
\includegraphics[scale=0.5]{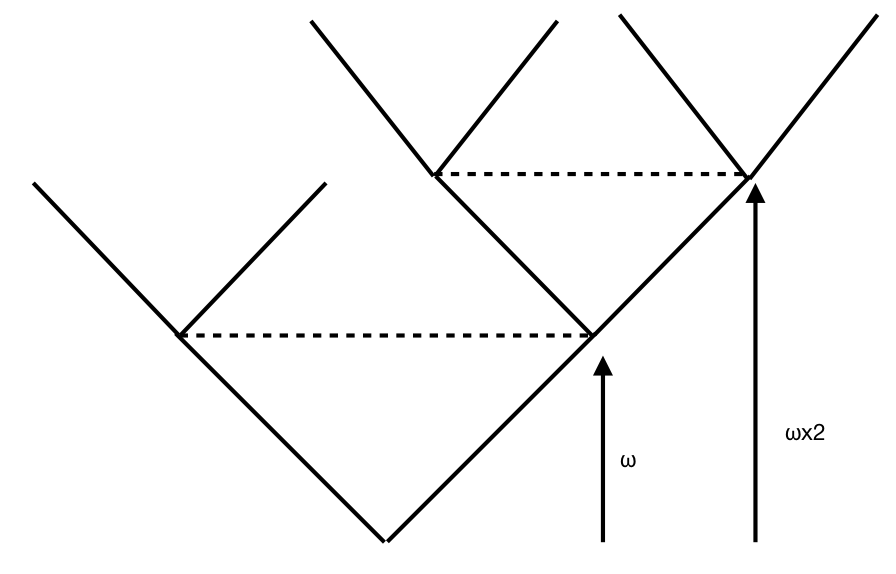}\\
\\
\emph{Figure 6: The hierarchy of clopen intervals of $\aleph$-sequences
produced by iterated bisection of a clopen interval. }\\
\\
But the condition that $x\in X$ can be decided by enumeration in
cardinal $<\aleph+1$ steps is equivalent to GCH as can be shown as
follows on the assumption that a function which decides $x\in X$
for all $x$ has $\aleph+1$ bits. This assumption follows from a
principle of information minimization since any function $f$ that
decides $x\in X$ cannot contain $\le\aleph$ bits as $f$ could be
represented by some binary $\le\aleph$ sequence, a potential member
of $X$. We can express this by means of a diagonal function $d(\left\lceil y\right\rceil ):=1-\left\lceil y\right\rceil (\left\lceil y\right\rceil )$
for $\left\lceil y\right\rceil $ a $\aleph$-bit code for a function
of $\aleph$-bits, and note that we get a contradiction if we put
$d:=\left\lceil y\right\rceil $ unless the number of bits in $d$
is greater than the number of bits in $\left\lceil y\right\rceil $.
It follows from a principle of information minimization that $f$
also contains $\aleph+1$ bits of information if $\aleph$ is an infinite
cardinal because, if we consider a generic $\aleph$-sequence $x$,
then there are many ordinals $\alpha$ of the same cardinality as
$\aleph$, and $\aleph+1$ bits would suffice to decide whether $x\in X$
or not by considering the least upper bound of $\alpha$ for generic
$\aleph\le\alpha<\aleph+1$-sequences that could be used to decide
$x\in X.$
\begin{thm}
\label{thm:GCH-is-equivalent}(Not provable in ZFC) GCH is equivalent
to\footnote{Strictly the inference from the information limitation principle to
GCH is probabilistic (true almost always) in cardinality terms rather
than logically necessary. } the assertion that the amount of information needed to decide the
relation $x\in X$ by an interleaved enumeration of $X$ or $2^{\aleph}-X$
is $<\aleph+1$, for any given binary $\aleph$-sequence x of length
at most cardinal $\aleph\ge\aleph_{0}$ and X has cardinality $\le2^{\aleph}$. 
\end{thm}

\begin{proof}
Assume that:
\begin{enumerate}
\item $\emptyset\subseteq X\subseteq2^{\aleph}$, 
\item \emph{$X$ }has cardinality $\aleph<c<2^{\aleph}$,
\item Any $x\in X$ is expressed as a binary sequence of length at most
cardinal $\aleph\ge\aleph_{0}$, and
\item The amount of information needed to decide the relation $x\in X$
by an interleaved enumeration of $X$ or $2^{\aleph}-X$ is $<\aleph+1$. 
\end{enumerate}
The proof is summarized in the tables below, where a $\checked$ means
that the option is possible and $\times$ means that the option is
impossible. \\

\begin{tabular}{|c|c|c|}
\hline 
 &
Enumerate $X$ &
Enumerate $2^{\aleph}-X$\tabularnewline
\hline 
\hline 
$x\in X$ &
$<c$ $\checked$ &
$2^{\aleph}$ $\times$\tabularnewline
\hline 
$x\notin X$ &
$c$ $\times$ &
$<2^{\aleph}$ $\checked$\tabularnewline
\hline 
\end{tabular}\\

\emph{Table 1: The number of steps to decide $x\in X$ by enumeration}\\
\\
\begin{tabular}{|c|c|c|c|}
\hline 
$<c$ &
Proof Ref. &
$c$ &
Proof Ref.\tabularnewline
\hline 
\hline 
$\aleph+1<c$ $\times$ &
1 &
$\aleph+1<c$ $\times$ &
4\tabularnewline
\hline 
$\aleph+1=c$ $\checked$ &
2 &
$\aleph+1=c$ $\times$ &
5\tabularnewline
\hline 
$\aleph+1>c$ $\times$ &
3 &
$\aleph+1>c$ $\times$ &
3\tabularnewline
\hline 
\end{tabular}\\
\\
\\
\begin{tabular}{|c|c|c|c|}
\hline 
$<2^{\aleph}$ &
Proof Ref. &
$2^{\aleph}$ &
Proof Ref.\tabularnewline
\hline 
\hline 
\textbf{$\aleph+1<2^{\aleph}$ }$\times$ &
1 &
$c<2^{\aleph}$ $\times$ &
8\tabularnewline
\hline 
\textbf{$\aleph+1=2^{\aleph}$ }$\checked$ &
6 &
$c<2^{\aleph}$ $\times$ &
8\tabularnewline
\hline 
\textbf{$\aleph+1>2^{\aleph}$ }$\times$ &
7 &
$c<2^{\aleph}$ $\times$ &
8\tabularnewline
\hline 
\end{tabular}\\
\\
\emph{Table 2: The possible cardinal relationships for the number
of steps in Table 1 and proof references}\\
\emph{}\\
Proof references:\\
\\
1. $x\in X$ would almost always be decided in $\ge\aleph+1$ bits
for a given enumeration of $X$, contradicting assumption d). \\
2. $\aleph+1=c$ is consistent with assumption d), as $x\in X$ would
be decided in $<c=\aleph+1$ steps by enumeration.\\
3. $\aleph+1>c$ contradicts assumption b) $\aleph<c$, as there would
be a cardinal strictly between $\aleph$ and $\aleph+1$.\\
4. $x\in X$ would almost always be decided in $>\aleph+1$ bits for
a given enumeration of $X$, contradicting assumption d). \\
5. $\aleph+1=c$ implies that $\aleph+1$ bits are needed to decide
$x\in X$ by enumerating all of $X$, which contradicts assumption
d).\\
6. \textbf{$\aleph+1=2^{\aleph}$ }is consistent with assumption d),
as $x\in X$ would be decided in $<2^{\aleph}=\aleph+1$ steps by
enumeration.\\
7. \textbf{$\aleph+1>2^{\aleph}$ }contradicts Cantor's theorem tha\textbf{t
$\aleph+1\le2^{\aleph}$.}\\
8. $c<\left|2^{\aleph}-X\right|=2^{\aleph}$ and therefore $x\in X$
could always be decided in $<2^{\aleph}$ steps by enumeration of
$X$.\\

We can conclude that if $x\in X$ then $c=\aleph+1$ and if $x\notin X$
then $\aleph+1=2^{\aleph}$. Using predicate logic\footnote{Existential elimination: for example, assume $(\exists x)(x\in X)$
and $(\forall x)(x\in X\rightarrow c=\aleph+1)$, then if $c\neq\aleph+1$
then by contraposition $(\forall x)(x\notin X)$ and hence $\neg(\exists x)(x\in X)$,
contradiction; hence $c=\aleph+1$.} we can conclude $(\exists x)(x\in X)\rightarrow c=\aleph+1$ and
$(\exists x)(x\in2^{\aleph}-X)\rightarrow\aleph+1=2^{\aleph}$. Since
both \emph{X} and $2^{\aleph}-X$ are not empty we can conclude that
$c=\aleph+1=2^{\aleph}$, which contradicts the assumption that $c<2^{\aleph}$.
GCH then follows.\\
\\
Conversely, assume GCH. Then if $x\in X$ then by GCH $x$ will be
enumerated in $<\left|X\right|\le2^{\aleph}=\aleph+1$ steps. While
if $x\notin X$ then $x$ will be enumerated in $<\left|2^{\aleph}-X\right|=2^{\aleph}=\aleph+1$
steps. In either case then $x\in X$ can be decided by enumeration
in $<\aleph+1$ steps, \emph{i.e.} in $<\aleph+1$ bits. 
\end{proof}
\begin{rem}
What this result shows that if the class of all pure sets $V$ is
considered to be a hierarchy of $\langle2^{\aleph},\oplus\rangle$-generalized
metric spaces, then GCH holds based on a principle of information
minimization. It is of course not true that the class of all pure
sets in $V$ needs to be a hierarchy of $\langle2^{\aleph},\oplus\rangle$-
generalized metric spaces, but it is a natural construction of $V$
based on a natural topology of sets.
\end{rem}

\section{Alternative Definition of Cardinality}

Having shown that there are models of ZFC in which Theorems 10 and
12 are true, we can now redefine cardinality to reflect cardinality
in this model (in which GCH is true). In terms of the normal definition
of cardinality, $2^{\aleph}-(\aleph+1)\times2^{\aleph}=\slashed{O}$
is not surprising; it simply says that $2^{\aleph}-max(\aleph+1,2^{\aleph})=2^{\aleph}-2^{\aleph}=\slashed{O}$.
However, the fact that $\aleph+1$ is the least cardinal number with
the property that $2^{\aleph}-(\aleph+1)\times2^{\aleph}=\slashed{O}$
is forced in the Baire topology suggests a modification to the definition
of cardinal number. Intuitively, the idea is that iterating the closed
nowhere dense set construction in a dense way on a nowhere dense set
will produce ever sparser nowhere dense sets, but the deletion of
$\aleph+1$ such nowhere dense sets in a dense way results in the
empty set. Cardinality in these terms then measures how sparse a set
can be before it ceases to exist. Or, in the spirit of the Baire Category
Theorem, cardinality measures how many negligible sets you need to
add together before a non-negligible set is formed. \\
\\
This argument suggests a change of definition of cardinality, namely
a set \emph{X} (represented as a tree \emph{T}) has cardinality $\aleph$
if $\aleph$ is the largest cardinal such that every dense (in the
sense that every non-empty open set of the tree has non-empty intersection
with the sequence), non-repeating sequence of (clopen) splitting subtrees
of \emph{T} of length $\aleph$ has an empty remainder after removal
of $\aleph$ subtrees in this sequence from \emph{T}, while it is
possible to construct a non-empty remainder after removal of any subsequence
of length $<\aleph$. This construction is always possible because
the number of splitting subtrees is the same as the number of nodes
and the branch length, and it is always possible to delete $<\aleph-$many
paths (\emph{i.e}. $\aleph$-sequences that are not branches) or branches
in a way that leaves a dense sequence of splitting subtrees.\footnote{The condition to delete $<\aleph-$many branches ensures that a tree
of cardinality $\aleph$ does not have cardinality $\aleph+1$ } This is so by the result in Section 4 that any set $X$ with a dense-in-itself
subset and of cardinality $\aleph<c\le2^{\aleph}$ has, under the
Baire topology, the property that $c-c\times(\aleph+1)=\slashed{O}$
and it is possible to force $c-c\times\aleph\ne\slashed{O}$ using
the closed nowhere dense set construction.\\
\\
In logical terms the change in definition of cardinality can be stated
as follows:
\begin{itemize}
\item $\left|T\right|=\alpha\leftrightarrow(P(T,\alpha)\wedge(\forall\gamma:Card(\gamma))(\gamma>\alpha\rightarrow\neg P(T,\gamma))$,
where
\item $\alpha$ is a cardinal, $Card(\alpha)$
\item \emph{T} is a binary tree with a root
\item $P(T,\beta):=(\forall\langle u_{\eta<\beta}\rangle:S(T,\langle u_{\eta<\beta}\rangle))(\bigcap_{\eta<\beta}u_{\eta}=\slashed{O})\wedge(\forall\gamma<\beta)$\\
$(\exists\langle u_{\delta<\gamma}\rangle:S(T,\langle u_{\delta<\gamma}\rangle))(\bigcap_{\delta<\gamma}u_{\gamma}\neq\slashed{O}$)
\item $S(T,\langle u_{\eta<\xi}\rangle):=NR(\langle u_{\eta<\xi}\rangle)\wedge D(T,\langle u_{\eta<\xi}\rangle)\wedge C(T,\langle u_{\eta<\xi}\rangle)$
\item $NR(\langle u_{\eta<\xi}\rangle):=(\forall\theta<\xi)(\forall\lambda<\xi)(\forall u_{\theta}\in\langle u_{\eta<\xi}\rangle)(\forall u_{\lambda}\in\langle u_{\eta<\xi}\rangle)$\\
$(u_{\theta}=u_{\lambda}\rightarrow\theta=\lambda)$ {[}non-repeating
sequence{]}
\item $D(T,\langle u_{\eta<\xi}\rangle):=(\exists\theta\leq\alpha)(\forall w_{\delta}\in\langle w_{\eta<\theta}\rangle:B(T,\langle w_{\eta<\theta}\rangle)(\exists u_{\kappa}\in\langle u_{\eta<\xi}\rangle)(w_{\delta}\cap u_{\kappa}\neq\slashed{O})$
{[}dense sequence{]}
\item $C(T,\langle u_{\eta<\xi}\rangle):=(\forall\rho<\xi)(u_{\rho}\neq\slashed{O}\wedge Clopensplit(T,u_{\rho}))$
{[}sequence of non-empty clopen splitting subtrees{]}
\item $B(T,\langle u_{\eta<\xi}\rangle):=C(T,\langle u_{\eta<\xi}\rangle)\wedge(T=\bigcup_{\delta<\xi}u_{\delta})$
{[}clopen basis for tree \emph{T}{]}
\item $Clopensplit(T,u):=(\exists x)(\exists\beta:Ord(\beta))(u=\{y\in T:y\ne x\wedge(\forall\gamma<\beta)(x_{\gamma}=y_{\gamma})\})$
\item $Ord(\alpha):=(\alpha=\emptyset\vee(\exists\beta<\alpha)(Ord(\beta)\wedge\alpha=\beta\cup\{\beta\})\vee(\alpha=\bigcup_{Ord(\beta):\beta<\alpha}\beta)$)
\item $Card(\alpha):=Ord(\alpha)\wedge(\neg\exists\beta<\alpha)[Ord(\beta))\wedge(\exists f:\beta\rightarrow\alpha)Sur(f)]$ 
\item $Sur(f):=Fn(f:X\rightarrow Y))\wedge(\forall y)(\exists x)(y=f(x))$
\item $Fn(f:X\rightarrow Y):=(\forall x\in X)(\exists y\in Y)(y=f(x))\wedge\forall x\in X)(\forall y\in Y)[x=y\rightarrow f(x)=f(y)]$
\end{itemize}
In the case of sets of natural numbers, a number $n$ can be represented
by a sequence $\langle1,\dots,1,0,\ldots\rangle$, \emph{i.e.} $n$
1s and then a terminal $\omega$-sequence of 0s. Removal of a dense
sequence of clopen subtrees is visually the removal of all of the
terminal $\omega$-sequence of 0s (and of course the initial sequence
of 1s) because in the discrete topology each set $\{n\}$, \emph{i.e.}
<1, Thus $\aleph_{0}-\aleph\times\aleph_{0}=\slashed{O}$ has solution
$\aleph=\aleph_{0};$ and $N$ has cardinality $\aleph_{0}$ and a
finite set with $n$ members has cardinality $n$, as before.\\
\\
It can be seen that according to the modified definition of cardinality
$2^{\aleph}=\aleph+1$ for all cardinals $\aleph\ge\aleph_{0}$, and
there are no cardinals $\aleph+1<\beth<2^{\aleph}$, \emph{i.e.} that
the Generalized Continuum Hypothesis is true in the sense of the new
definition of cardinality.

\section{Conclusions}

There is a natural way to measure size of sets, which is given by
how many clopen intervals need to be deleted in a dense way from a
binary tree of binary $\aleph$-sequences before the empty set results
(or a countable set of isolated points). This definition of cardinality
works well for a set universe which satisfies the Axiom of Choice,
when all sets of size $\le2^{\aleph}$ are sets of binary $\aleph$-sequences.
The price to be paid for the use of a Baire topology (in which clopen
intervals exist) is that the Baire topology is pathological in several
respects: clopen sets are totally disconnected\footnote{A topological space $X$ is \emph{totally disconnected} if for every
two points $x,y\in X$ such that $x\ne y$ there are disjoint open
sets $O_{1}$ and $O_{2}$ such that $x\in O_{1}$, $y\in O_{2}$
and $O_{1}\cup O_{2}=X$. } by definition, and Baire topological spaces, such as clopen interval
$([0,1])[\aleph]$, comprise a  set of binary $\aleph$-sequences
are not compact or metrizable (for $\aleph>\aleph_{0}$). That said,
since all sets can be regarded as $\aleph$-tuples of real numbers,
there is a natural generalized metric and in a $\aleph$-sequentially
complete topological space, every strictly nested decreasing $\aleph$-sequence
of clopen intervals is a set with a single element. The class of sets
is then quite well behaved under the assumption of the Axiom of Choice,
although this good behaviour does not extend to the properties of
sets that can be created using the Axiom of Choice (see \cite{key-5}
for a selection of such sets).

\end{document}